\documentclass[12pt]{amsart}
 \usepackage{amssymb,amsfonts,amsthm,amscd,verbatim, color}
\usepackage[square, comma, sort&compress, numbers]{natbib}
 \usepackage{graphicx}
 \usepackage{bm,comment}
\newtheorem{thm}{Theorem}[section]
\newtheorem{lem}{Lemma}[section]
\newtheorem{cor}{Corollary}[section]
\newtheorem{pro}{Proposition}[section]

\newtheorem{ex}[thm]{Example}

\theoremstyle{definition}

\theoremstyle{remark}
\newtheorem{rem}{Remark}[section]
\numberwithin{equation}{section}

\def\R{{\Bbb R}}
\let\cal=\mathcal

\begin{document}

\title[Systematic Measures of Biological Networks, Part I]{Systematic Measures of Biological Networks,
Part I: Invariant measures and Entropy}
\author[Y. Li] {Yao Li}
\address{Y. Li: Department of Mathematics and Statistics, University
  of Massachusetts Amherst, MA, 01003, USA and Courant Institute of
  Mathematical Sciences, New York University, NY, 10012, USA}
\email{yaoli@math.umass.edu}

\author[Y. Yi]{Yingfei Yi}
\address{Y. Yi: Department of Mathematical
\& Statistical Sci, University of Alberta, Edmonton, Alberta,
Canada T6G 2G1 and  School of Mathematics, Jilin University, Changchun
130012, PRC}
\email{yingfei@ualberta.ca,yi@math.gatech.edu}

\thanks {The research was partially supported by DMS1109201.  The
  second author is also partially supported by NSERC discovery grant
  1257749, a faculty development grant from University of Alberta, and
  a Scholarship from Jilin University. }



\begin{abstract} This paper is Part I of a two-part series devoting to the
study of systematic measures in a complex biological network modeled by a
system of ordinary differential equations.  As the mathematical
complement to our previous work \cite{li2012quantification} with
collaborators, the series aims at establishing a mathematical
foundation for characterizing three important  systematic measures:
degeneracy, complexity and robustness,  in such a biological network and
studying connections among them. To do so, we consider in Part I
stationary measures of a Fokker-Planck equation generated from small
white noise perturbations of a dissipative system of ordinary
differential equations. Some estimations of concentration of
stationary measures of the Fokker-Planck equation in the vicinity of
the global attractor are presented. Relationship between
differential entropy of stationary measures and dimension of the
global attractor is also given.

\end{abstract}

\maketitle

\section{Introduction}
The concept of modular biology has been proposed and extensively
investigated in the past several decades. In a complex
biological network,  modules in cells are created by interacting
molecules that function in a semi-autonomous fashion and they  are
functionally correlated. To better understand the interactions
between modules in a complex biological network, it is necessary to
quantitatively study systematic properties such as degeneracy,
robustness, complexity, redundancy, and evolvability.

Emerged from early studies of brain functions
\cite{edelman1978mindful}, notions of degeneracy and complexity were
first introduced in neural networks in \cite{tononi1994measure},
and the robustness was studied in \cite{kitano2007towards,
kitano2004biological} for systems with performance functions.
Roughly speaking, in a cellular network or a neural network degeneracy
measures the capacity of elements that are structurally different to
perform the same function,  structural complexity measures the
magnitude of functional integration and local segregation of
sub-systems, and
  the robustness
measures the capacity of performing similar function under
perturbation. These systematic measures are known to be closely
related. Indeed, it has already been observed via numerical
simulations for neural networks that high degeneracy not only yields
high robustness, but also it is accompanied by an increase in
structural complexity \cite{tononi1999measures}.


As increasing biological phenomena were being observed, quantitative
studies of systematic measures in biological networks were also being
conducted. For instance,  numerical
simulations revealed connections between degeneracy and complexity in
artificial chemistry binding systems \cite{clark2011degeneracy};
 and also conclude that degeneracy underlies the
presence of long range correlation in complex networks
\cite{delignieres2011degeneracy, delignieres2013degeneracy}.
Features like regulation and robustness of biochemical networks of
signal transduction have also been studied quantitatively in
\cite{kitano2007towards, rizk2009general}. Degeneracy, complexity
and robustness were  quantified for neural networks by making use
of testing noise injections into the networks in
\cite{tononi1999measures}. However, it was later remarked in the
review article \cite{edelman2001degeneracy} that ``degeneracy is a
ubiquitous property of
 biological systems at all levels of organization, the concept has not yet been fully incorporated
 into biological thinking, $\dots$ because of the lack of a general evolutionary framework for the concept
 and the absence of a theoretical analysis".
Recently, quantification of degeneracy, complexity and robustness in
biological networks modeled by systems of ordinary differential
equations  was made  in the authors' joint work
\cite{li2012quantification} with Dwivedi, Huang and Kemp. The goal
of that study is precisely to extend the concept of degeneracy to an
evolutionary biological network and to establish its connections
with complexity and robustness.

The present work, consisting of two parts,  serves as the
mathematical complement of our previous work
\cite{li2012quantification},  aiming at establishing a mathematical
foundation of degeneracy, complexity and robustness in a complex
biological network modeled by a system of ordinary differential
equations.

This mathematical foundation is based on the theory of stochastic
differential equations. In particular, considering testing white
noise  perturbations to a  biological network is important in the
quantification of its systematic measures  because characterizations
of degeneracy and complexity rely on the functional connections
among modules of the network and such connections can be activated
by injecting external noises,  similarly  to the case of neural
systems \cite{tononi1999measures, li2012quantification}.

To be more precise, consider a biological network modeled by the following
system of ordinary differential equations (ODE system for short):
\begin{equation}
   \label{ODE1}
   x' = f(x),\;\;\;x\in \R^n,
\end{equation}
where $f$ is a $C^1$ vector field on $\R^n$, called {\em drift
field}. Under additive white noise perturbations $\sigma dW_{t}$, we
obtain the following system of stochastic differential equations
(SDE system for short):
\begin{equation}
   \label{SDE1}
   \mathrm{d}X = f(X) \mathrm{d}t + \epsilon \sigma(x) \mathrm{d}W_{t},\; \;\; X\in \R^n,
\end{equation}
where  $W_t$ is the standard $m$-dimensional Brownian
 motion,  $\epsilon$ is a small parameter lying in an interval $(0,\epsilon^*)$, and  $\sigma$,
 called an {\em noise matrix}, is
 an  $n\times m$ matrix-valued, bounded, $C^1$ function on $\R^n$  for some positive integer $m\ge n$,
 such that $\sigma(x)\sigma^{\top}(x)$ is everywhere
positive definite in $\mathbb R^n$.  We denote the collection of
such noise matrices by $\Sigma$. Under certain dissipation condition
(e.g., the existence of Lyapunov function corresponding to
\eqref{SDE1} assumed in this paper), the SDE system \eqref{SDE1}
generates a diffusion process in $\R^n$ which admits a transition
probability kernel $P^{t}(x, \cdot)$, $t\ge 0$, $x\in\R^n$,  such
that for each $x \in \mathbb{R}^{n}$, $P^{t}(x, \cdot)$ is a
probability measure and for each Borel set $B \subset
\mathbb{R}^{n}$, $P^{t}(x, B)$  measures the probability of the
stochastic orbit starting from $x$
 entering the set $B$ at  time $t$. An {\em invariant probability
measure of the diffusion process} associated with \eqref{SDE1} is
the left invariant of $P^{t}(x, \cdot)$ such that
\[
(\mu_{\epsilon} P^{t})(\cdot) = \int_{\R^n} P^{t}(x, \cdot)\mathrm{d}  \mu_\epsilon=\mu_\epsilon(\cdot),\qquad
t\ge 0 \,.
\]
An invariant probability measure associated with \eqref{SDE1} is
necessarily  a stationary measure of the Fokker-Planck
equation associated with \eqref{SDE1}.   In general, a
stationary measure can be regarded as a ``generalized invariant
measure" if the diffusion process fails to admit an invariant
measure.


 By injecting external
fluctuation $\epsilon\sigma dW_{t}$ into the network \eqref{ODE1}, the connections
among different modules of the network are activated.  Assuming the existence of a stationary measure
$\mu_\epsilon$ of the Fokker-Planck equation  associated with
\eqref{SDE1} for each $\epsilon\in (0,\epsilon^*)$, the mutual
information  among any two modules (coordinate subspaces) $X_1,X_2$ can be
defined using the margins $\mu_{1}$, $\mu_{2}$   of $\mu_{\epsilon}$
with respect to $X_{1}, X_{2}$, respectively. Such mutual
information can then be used to quantify degeneracy and complexity, and further to examine their connections with dynamical
quantities like robustness.

Such a mathematical foundation for degeneracy, complexity, and
robustness in a biological network
 relies on a
quantitative understanding of the
 stationary measures $\mu_{\epsilon}$ particularly with respect to their
 concentrations. This is in fact the main subject of
 this part of the series.

 A well-known approach to study the invariant probability measure is
 the classical large-deviation theory (or Freidlin-Wentzell
 theory). The probability that the trajectory of \eqref{SDE1} stays in the
 neighborhood of any absolutely continuous function can be calculated
 explicitly by Girsanov's theorem. This leads to some estimates of
 tails of stationary measures, or the first exit
 time of a stochastic orbit (see
 e.g. \cite{dembo2009large,FW, kifer1988random}). For instance, it is shown in
 \cite{FW} that for any set $P\subset
\mathbb{R}^{n}$ that does not intersect with any attractor of
\eqref{ODE1}, there exists a constant $V_{0}>0$  such that
\begin{equation}\label{QP}
   \lim_{\epsilon\rightarrow 0}\epsilon^{2}\log \mu_{\epsilon}(P) =
   -V_{0}.
\end{equation}
 In particular, the limit
\begin{equation}
\label{QPP}
  -\lim_{\epsilon\rightarrow 0}\epsilon^{2}\log
  \frac{\mathrm{d}\mu_{\epsilon}(x)}{\mathrm{d}x} := V(x)\,,
\end{equation}
if exists, is called the {\it quasi-potential function}.

One limitation of the large deviation theory is that usually it can only
estimate the probability of rare events, which corresponds to the
 tails of stationary measures.  In many applications, more
 refined analysis is based on the assumption that $\mu_{\epsilon}$ can
be approximated by a Gibbs measure, i.e., $\mu_{\epsilon} $ admits a
density function $u_{\epsilon}$ such that
\begin{equation}
  \label{wkb}
u_{\epsilon}(x) \approx \frac{1}{K}e^{-V(x)/\epsilon^{2}} \,,
\end{equation}
where $V(x)$ is the quasi-potential function \cite{risken1996fokker,
  grasman1999asymptotic, talkner1987mean, schuss2009theory}. However,
this assumption is difficult to verify in general as it requires high
regularity of the quasi-potential function. Rigorous results are only known for
some simple systems \cite{day1994regularity, day1985some, ludwig1975persistence}.

To understand connections among degeneracy, complexity, and
robustness, we need to measure the effects of stochastic perturbations
in \eqref{SDE1}  at the same order as $\epsilon$. To  make
such estimation rigorously, we will adopt the level set
method recently introduced in \cite{huang1,huang2} for stationary
probability measures of the Fokker-Planck equation associated with
\eqref{SDE1} (see Section ~\ref{background}).

In this part of the series, we will mainly apply the level set
method to obtain estimates on the concentrations of invariant
measures $\mu_\epsilon$ with respect to a fixed $\sigma\in\Sigma$.
Our main results of the paper lie in the following three categories.
\medskip

\begin{enumerate}
\item[a)] {\em Concentration in the vicinity of the global attractor $\cal A$}:
 We will show in  Theorem ~\ref{accurate} that  for any  $0<\delta\ll
 1$
there exists a constant $M>0$ such that
\begin{displaymath}
   \mu_{\epsilon}(\{x:\, \mathrm{dist}(x,\mathcal{A}) \leq M \epsilon
   \}) \geq 1-\delta.
\end{displaymath}
We will also show in Theorem ~\ref{levelalpha} that for any $\alpha
\in (0,1)$,
$$
  \lim_{\epsilon\rightarrow 0} \mu_{\epsilon}(\{x :\,
  \epsilon^{1+\alpha} \leq \mathrm{dist}(x, \mathcal{A}) \leq
  \epsilon^{1-\alpha}\}) = 1 \,.
$$

\item[b)] {\em Mean square displacement}:
 We will show in Theorem ~\ref{MSDbounds}  under certain conditions that there
 exist constants
$V_{1}, V_{2} > 0$ such that
\begin{displaymath}
   V_{1}\epsilon^{2}\leq V(\epsilon) \leq V_{2}\epsilon^{2},
\end{displaymath}
where
\begin{equation*}
   V(\epsilon) = \int_{\mathbb{R}^{n}}
   \mathrm{dist}^{2}(x,\mathcal{A}) d\mu_{\epsilon}(x).
\end{equation*}

\item[c)] {\em Entropy-dimension relationship}: We will show in Theorem
~\ref{EntDimThm} that if the global attractor $\cal A$ is regular,
then
\begin{equation*}
   \liminf_{\epsilon\rightarrow
     0}\frac{\mathcal{H}(\mu_{\epsilon})}{\log \epsilon} \geq n-d
\end{equation*}
where $\mathcal{H}(\mu_{\epsilon})$ is the differential entropy of
$\mu_\epsilon$ and  $d$ is the Minkowski dimension of $\cal A$.
\end{enumerate}
\medskip

The paper is organized as follows. Section ~\ref{background} is a
preliminary section in which we mainly review some results and the
level set method from \cite{huang1, huang2, huang5} on Fokker-Planck
equations.  Concentrations of stationary measures are
studied in Section ~\ref{concentration}. We derive the
entropy-dimension relationship in Section 4.

\section{Preliminary}\label{background}
In this section, we will review some notions and known
results about dissipative dynamical  systems and Fokker-Planck
equations including global attractors, Lyapunov functions, and the
existence and uniqueness of stationary measures. We will also recall
a Harnack inequality to be used later.

 \subsection{Dissipation and global attractor}

We note that the system \eqref{ODE1} generates a local flow on
$\R^n$, which we denote by $\phi^t$. For $B \subset \mathbb{R}^{n}$, we denote $\phi^{t}(B) =
\{\varphi^{t}(x) \ : \, \ x \in B\}$.  A set $A\subset\R$ is said to be
{\it invariant} with respect to \eqref{ODE1} or $\varphi^t$ if $\varphi^t$
extends to a flow on $A$ and  $\phi^{t}(A) \subset A$ for any  $t
\in\R$.

System \eqref{ODE1} or $\phi^t$ is said to be {\em dissipative}  if
$\varphi^t$, $t\ge 0$, is a positive semi-flow on $\R^n$ and  there
exists a compact subset $K$ of $\R^n$ with the property that for any
$\xi\in \R^n$ there exists a $t_0(\xi)>0$ such that
$\varphi^t(\xi)\in K$ as $t\ge t_0(\xi)$. It is well-known that if
$\varphi^t$ is dissipative, then it must admit a {\em global attractor}
$\mathcal A$, i.e., $\mathcal{A}$  is a compact subset of
$\mathbb{R}^{n}$ which attracts any bounded set in $\mathbb{R}^{n}$
in the sense that $\lim_{t \rightarrow +\infty}
\mathrm{dist}(\varphi^{t}(K), \mathcal{A}) = 0$ for every bounded
set $K \subset \mathbb{R}^{n}$, where $\mathrm{dist}(A,B)$ denote
the Hausdorff semi-distance from a bounded set $A$ to a bounded  set
$B$ in $\mathbb R^n$. The global attractor $\mathcal A$ of $\varphi^t$,
if exists,  must be unique and invariant with respect to  $\varphi^t$.
In fact, $\varphi^t$ is dissipative if and only if it is a semi-flow admiting a
global attractor. Moreover, dissipation of $\varphi^t$ can be
guaranteed by the existence of a {\em Lyapunov function} $U$ of
\eqref{ODE1}, i.e., $U\in C^1(\R^n)$ is a non-negative function such
that $U(x)<\sup_{x\in\R^n} U(x)$, $x\in \R^n$,  and there exist a
compact set $K\subset\R^n$ and a constant $\gamma>0$, called a {\em
Lyapunov constant}, such that
$$
  f(x) \cdot \nabla U(x) \leq - \gamma, \;\qquad x \in
  \R^n\setminus K.
$$

The global attractor $\mathcal{A}$ of \eqref{ODE1} is said to be a
{\it strong attractor} if there is a connected open neighborhood
$\mathcal{N}$ of $\cal A$ with $C^{2}$ boundary, called {\it
isolating neighborhood}, such that i) $\omega( \mathcal{N}) =
\mathcal{A}$ and ii) $f(x) \cdot \nu(x) < 0$ for each $x \in
\partial \mathcal{N}$, where $\nu(x)$ is the  outward normal vector
of $ \mathcal{N}$ at $x$ and $\omega(B) := \bigcap_{\tau \ge 0
 }
\overline{\{ \varphi^{t}(B)\ : \, \ t\geq
  \tau \}}$ is the {\it $\omega$-limit set} of a bounded set $B\subset
  \mathbb{R}^{n}$. It is clear that $\cal A$ is a strong
attractor of \eqref{ODE1}  if there exists a {\em strong Lyapunov
function}  in a connected  open set $S \subseteq \R^n$ containing
$\cal A$, i.e.,
  $\nabla U(x)\ne 0$, $x\in  S\setminus \cal A$, and there is a constant $\gamma_{0}>0$  such that
$$
  f(x) \cdot \nabla U(x) \leq - \gamma_{0} |\nabla U(x)|^{2}, \;\qquad x \in
  S\setminus \cal A.
$$
 We again refer the constant $\gamma_0$ above to as the {\it
Lyapunov constant} of $U$.

\medskip

\subsection{Fokker-Planck equation and stationary measures}
If the transition probability kernel $P^{t}(x, \cdot)$ of the SDE
system \eqref{SDE1} admits a probability density function $p^{t}(x, z)$,
 i.e.,
$$
  P^{t}(x, B)= \int_{B} p^{t}(x,  z) \mathrm{d}z
$$
for any Borel set $B \subset \mathbb{R}^{n}$, then for any measurable, non-negative function $\xi(x)$ with
$\int_{\mathbb{R}^{n}} \xi(x) \mathrm{d}x = 1$, $ u_{\epsilon}(x,t) =
\int_{\mathbb{R}^{n}} p^{t}(z,  x) \xi(z) \mathrm{d}z $
characterizes the time evolution of the probability density function.
Formally, $u_{ \epsilon}(x,t)$ satisfies the following Fokker-Planck equation
(FPE for short) :
\begin{equation}
   \label{FPE1}
  \left\{\begin{array}{l} \displaystyle{ \frac{\partial
            u_{\epsilon}(x,t)}{\partial t}=\frac{1}{2}{{\epsilon
}^{2}}\sum\limits_{i,j=1}^{n}{\partial_{ij}(a_{ij}(x) u_{\epsilon}(x,t) )-
\sum_{i=1}^{n}\partial_{i}(f(x) u_{\epsilon}(x,t)
   )} := L_{\epsilon}u_{\epsilon}(x,t),}\\
  \displaystyle{ \int_{\R^n} u_{\epsilon}(x,t) \rm dx=1 },\end{array}\right.
\end{equation}
where $(a_{ij}(x)) := A(x) := \sigma (x){{\sigma }^{\top}}(x)$ is an
$n\times n$, everywhere positive definite, matrix-valued $C^1$
function, called {\em diffusion matrix}. The operator
$L_{\epsilon}$ is called {\em Fokker-Planck operator}.

Among solutions of the Fokker-Planck equation, of particular
importance are the {\em stationary solutions}.  For any
connected open subset $S \subset \mathbb{R}^{n}$, stationary
solutions on $S$ satisfy the stationary Fokker-Planck equation
\begin{equation}
   \label{ssFPE}
   L_{\epsilon} u_{\epsilon}(x,t) = 0\,,\quad \int_{S}  u_{\epsilon}(x,t) \mathrm{d}x = 1\,,\qquad
   u
   \geq 0 \,.
\end{equation}

More generally, on any connected open subset $S \subset
\mathbb{R}^{n}$, a {\em
  stationary measure} of the Fokker-Planck
equation is a probability measure $\mu_{\epsilon}$ satisfying
\begin{displaymath}
   \int_{S}\mathcal{L}_{\epsilon}h(x) \mathrm{d}\mu_{\epsilon} = 0 \,,\quad
   \forall h(x) \in C_{0}^{\infty}( S) \,,
\end{displaymath}
where
\[
   \mathcal{L}_{\epsilon} =
   \frac{1}{2}\epsilon^{2}\sum_{i,j=1}^{n}a_{ij}(x)\partial_{ij} +
   \sum_{i=1}^{n} f_{i}(x)\partial_{i}
\]
is the {\it adjoint Fokker-Planck
  operator}.

 If $u_\epsilon$ is a stationary solution of the
Fokker-Planck equation \eqref{FPE1}, then $u_\epsilon {\rm d} x$ is
clearly a stationary measure. Conversely, it follows from the
regularity theory of Fokker-Planck equation
\cite{bogachev2001regularity} and standard regularity theory of
elliptic equation that a stationary measure of Fokker-Planck
equation \eqref{FPE1}  must admit a density function which is a
strictly positive, classical stationary solution of the
Fokker-Planck equation. Note that a classical solution
  means a solution that has enough regularity to be plugged into the
  original differential equation.

An invariant measure of the diffusion process generated by
\eqref{SDE1}, or equivalently, of the transition probability kernel
$P^{t}$, is necessarily a stationary measure of the corresponding
Fokker-Planck  equation \eqref{FPE1}. The converse needs not be true
in general. However, stationary measures considered in this paper
are in fact invariant measures of the diffusion process generated by
\eqref{SDE1}.


The existence and estimates of stationary measures of \eqref{FPE1}
are related to Lyapunov-like functions associated with it.  For the
sake of generality, we consider a connected open set  $S \subseteq
\mathbb{R}^{n}$. A non-negative function $U\in C(S)$ is said to be a
{\em compact function} if (i) $U(x)< \rho_M$, $x\in S$; and (ii)
$\lim_{x \rightarrow
\partial S} U(x)=\rho_M$, where $\rho_{M} =
\sup_{x \in S} U(x)$ is called the {\em  essential upper bound of
$U$}. In the case $S=\R^n$, $x \rightarrow
\partial S$ simply means that $x\to\infty$. It is obvious that
Lyapunov and strong Lyapunov functions defined in the previous
subsection are all compact functions on $\R^n$.

For a compact function defined on $S$ and for
each $\rho\in  [0,\rho_M)$, we  denote $\Omega_\rho(U) = \{x\in S: U(x)<\rho\}$
 as the {\it$\rho$-sublevel set} of $U$ and  $\Gamma_\rho(U) = \{x\in
 S: U(x)=\rho\}$ as the  {\it $\rho$-level set} of $U$.

Let $U$ be a compact
$C^{2}$ function on a connected open set $S \subset \mathbb{R}^{n}$
 with upper bound $\rho_M$. For a fixed $\epsilon
 \in (0,\epsilon^*)$, $U$ is
 called  a {\em Lyapunov function associated with \eqref{FPE1}} (on $S$)  if there are constants
$\rho_m,\gamma>0$, referred to as  {\em an essential lower
bound, the Lyapunov constant of $U$}, respectively,
  such that
\[
 \mathcal{L}_{\epsilon} U(x) < -\gamma, \;\qquad x\in S\setminus
  \Omega_{\rho_m}(U).
\]
 $U$  is called a {\it weak Lyapunov function} (on $S$)
associated with equation \eqref{FPE1} if there is a constant
$\rho_{m}$, still referred to as an essential lower bound of $U$,
such that
\[
  \mathcal{L}_{\epsilon} U(x) \leq 0, \;\qquad x\in S\setminus
  \Omega_{\rho_m}(U) \,.
\]

If $U(x)$ is a Lyapunov function (resp. weak Lyapunov function )
associated with \eqref{FPE1} for each $\epsilon \in (0, \epsilon^*)$
and the essential lower bound and Lyapunov constant are independent
of $\epsilon$, then $U(x)$ is called a {\it uniform Lyapunov
function} (resp. {\it uniform weak Lyapunov function}) associated
with the family \eqref{FPE1} on $(0, \epsilon^*)$.

\medskip

It is easy to see that a uniform Lyapunov function associated with
the family \eqref{FPE1}  on $(0, \epsilon^*)$ must be a Lyapunov
function for the ODE system \eqref{ODE1}.  Consequently, if the
family \eqref{FPE1} on $(0, \epsilon^*)$ admits a uniform Lyapunov
function, then the ODE system \eqref{ODE1} must be dissipative and
hence admits a global attractor.
\medskip

There has been extensive studies on the  existence and uniqueness of
stationary measures of Fokker-Planck equation \eqref{FPE1} (see
 \cite{bogachev2009elliptic, bogachev2012positive, huang2} and
references therein). While stationary measures of a Fokker-Planck
equation in a bounded domain of $\R^n$ always exist, the existence
of such in the entire space (i.e. $S = \mathbb{R}^{n}$) necessarily
require certain dissipation conditions at infinity which is more or
less equivalent to the existence of a Lyapunov function.

\medskip

The following theorem follows from the main result of
\cite{bogachev2012positive, huang2} and the standard regularity theory
of elliptic equations.

\medskip

\begin{thm}\label{exist1}  If the family ${\cal L}_\epsilon$, $\epsilon\in (0,\epsilon^*)$,  admits
 an unbounded  uniform Lyapunov function, then for each $\epsilon\in (0,\epsilon^*)$, the corresponding
  Fokker-Planck
    equation \eqref{FPE1} has a unique
    stationary measure $\mu_\epsilon$ on $\mathbb{R}^{n}$. Moreover,
    ${\rm d}\mu_\epsilon(x)=u_\epsilon(x){\rm d}x$  for a classical stationary
    solution $u_{\epsilon}$ of \eqref{FPE1}.
\end{thm}
\medskip

\subsection{Level set method and measure estimates} The following two theorems are the main ingredient of
the level set method introduced in \cite{huang1}.

\begin{thm} {\rm (Integral identity, Theorem 2.1, \cite{huang1})} \label{lem23}
For a given $\epsilon\in (0,\epsilon^*)$,
let $u=u_\epsilon$ be a stationary solution of  \eqref{FPE1}. Then
for any Lipschitz domain $S \subset\R^n$ and a function $F \in
C^{2}(S)$ having constant value on $\partial S$,
\begin{eqnarray}
   \label{const-bd}
  \int_{S} ({\cal L}_\epsilon F(x))
  u(x)\mathrm{d}x&=&\int_{S}  \left (\sum_{i,j = 1}^{n} \frac{1}{2}\epsilon^{2}a_{ij}(x)\partial^{2}_{ij}F(x) +
   \sum_{i=1}^{n}f_{i}(x)\partial_{i}
   F(x)\right )u(x)\mathrm{d}x \\\nonumber
&=& \int_{\partial
     S} \left (\sum_{i=1}^{n}\sum_{j=1}^{n} \frac{1}{2}\epsilon^{2}a_{ij}(x)\partial_{i} F(x)
   \nu_{j}\right )u(x) \mathrm{d}s
\end{eqnarray}
where $\{\nu_{j}\}_{j=1}^{n}$ denotes the unit outward normal
vectors.
\end{thm}

 In applying Theorem ~\ref{lem23} to study stationary
solutions of a Fokker-Planck equation, one typically considers $F$
as a Lyapunov function $U(x)$ and $S$ as a sublevel set
$\Omega_{\rho}(U) = \{x\in \R^n:  U(x) < \rho \}$. When $\nabla
U(x)\ne 0$ on the level set $\Gamma_\rho(U) = \{x\in \R^n: U(x) =
\rho \}$, we note that  $\partial \Omega_\rho(U)=\Gamma_\rho(U)$.
\medskip

\begin{thm} {\rm (Derivative formula, Theorem 2.2, \cite{huang1})}\label{deri}
Let $\mu$ be a Borel probability measure with density function $u\in
C(\R^n)$ and $U$ be a $C^1$ compact function on $\R^n$ such that
$\nabla U(x)\ne 0$, $x\in \Gamma_\rho(U)$ for all $\rho$ lying in an
interval $(\rho_1,\rho_2)$. Then
$$
\frac{\partial }{\partial\rho} \int_{\Omega_\rho(U)} u(x)~{\rm d} x=
\int_{\Gamma_{\rho}(U)} \frac{u(x)}{|\nabla U(x)|} ~{\rm d} s,\qquad
\rho\in (\rho_1,\rho_2).
$$
\end{thm}

\medskip

Let $\mu_{\epsilon}$ be a stationary measures of the Fokker-Planck
equation \eqref{FPE1}. Then as shown in \cite{huang1, huang2}, Theorems
~\ref{lem23},~\ref{deri} yield the following estimates concerning
$\mu_\epsilon$ in the presence of a Lyapunov function.

\medskip

\begin{lem} {\rm (Theorem A (b),  \cite{huang1})}
\label{lem41} Assume that \eqref{FPE1}  admits a
Lyapunov function $U$ with essential lower, upper bound $\rho_{m}$,
$\rho_M$, respectively,  that satisfies $\nabla U(x) \neq 0$,  $x
\in \Gamma_{\rho}$ for almost every $\rho\in [\rho_m,\rho_M)$. Then
for any function $H(\rho)\in L^{1}_{loc}([\rho_{m},\rho_{M}))$ with
\begin{displaymath}
   H(\rho)\geq
   \frac{1}{2}\epsilon^{2}\sum_{i,j=1}^{n}a_{ij}(x)\partial_{x_{i}}U(x)\partial_{x_{j}}U(x),\quad
   x \in \Gamma_{\rho},
\end{displaymath}
one has
\begin{displaymath}
   \mu_{\epsilon}(\Omega_{\rho_{M}}(U)\setminus \Omega_{\rho}(U)) \leq
   e^{-\gamma\int_{\rho_{m}}^{\rho}\frac{1}{H(t)}dt},\qquad \rho\in
   [\rho_m,\rho_M),
\end{displaymath}
where $\gamma >0$ is the Lyapunov constant of $U$.
\end{lem}

\medskip

\begin{lem} {\rm (Theorem A (c), \cite{huang1})}
\label{lem33} Assume that \eqref{FPE1} admits a weak
Lyapunov function $U$ in a connected open set $S\subseteq
\mathbb{R}^{n}$ with essential lower, upper bound $\rho_m,\rho_M$,
respectively. Also assume that $(a_{ij})$ is everywhere positive
definite in $S$. Then for any two positive continuous functions
$H_{1}(\rho), H_{2}(\rho)$ satisfying
$$
  H_{1}(\rho)\leq \frac{1}{2}\epsilon^{2}\sum_{i,j = 1}^{n}a_{ij}(x)\partial_{i}
  U(x)\partial_{j} U(x) \leq H_{2}(\rho),\quad x \in
  \Gamma_{\rho}(U),
$$
one has
$$
  \mu_{\epsilon}(\Omega_{\rho_{M}}(U) \setminus \Omega_{\rho_{m}}(U))
 \leq \mu_{\epsilon}(\Omega_{\rho}(U) \setminus
 \Omega_{\rho_{m}}(U))e^{\int_{\rho}^{\rho_{M}} \frac
 1{\tilde{H}(s)}
   \mathrm{d}s},\qquad \; \rho\in (\rho_m,\rho_M),$$
where $\tilde{H}(\rho) =
H_{1}(\rho)\int_{\rho_{m}}^{\rho}H^{-1}_{2}(s) \mathrm{d}s$.
\end{lem}

\subsection{Hanack inequality} We recall the following Harnack
inequality from  \cite{gilbarg2001elliptic}.

\begin{lem}
\label{harnack} Consider an elliptic operator
$$
  Lu(x) = \sum_{i,j=1}^{n}\partial_{i}(a_{ij}(x)\partial_{j}u(x))
  +\sum_{i=1}^{n}\partial_{i}(b_{i}(x)u(x)) +
  \sum_{i=1}^{n}c_{i}(x)\partial_{i}u(x) + d(x)u(x)
$$
in a domain $\Omega \subset \mathbb{R}^{n}$. Let $\lambda$ and
$\Lambda$ be two constants depend on matrix $\{a_{ij}(x)\}$ such
that
$$
  \sum_{i,j=1}^{n}a_{ij}(x)\zeta_{i}\zeta_{j} \geq \lambda |\zeta|^{2}
$$
and
$$
  \sum_{i,j=1}^{n}|a_{ij}(x)|^{2} \leq \Lambda^{2} \,.
$$
Let $\nu$ be a constant such that
$$
  \lambda^{-2}\sum_{i=1}^{n}(|b_{i}(x)|^{2} + |c_{i}(x)|^{2} +
  \lambda^{-1}|d(x)|) \leq \nu^{2} \,.
$$
Then for any ball $B_{4R}(y) \subset \Omega$, we have
$$
  \sup_{x \in B_{R}(y)} u(x) \leq C
  \inf_{x \in B_{R}(y)} u(x)
$$
where $C \leq C_{0}(n)^{(\Lambda/\lambda + \nu R)}$.
\end{lem}

\section{Concentration of stationary measures}\label{concentration}

We make the following standard hypothesis:

\medskip
\begin{itemize}
\item[{\bf H$^0$)}] System \eqref{ODE1} is dissipative
 and there exists a strong Lyapunov function
$W(x)$  with respect to an isolating neighborhood $S:=\mathcal{N}$
of the global attractor $\mathcal{A}$ such that
\begin{displaymath}
   W(x) \geq L_{1} \mathrm{dist}^{2}(x, \mathcal{A}),\;\; x \in \mathcal{N}
\end{displaymath}
for some $L_{1} > 0$.
\end{itemize}

\medskip

\begin{rem} When $\mathcal{A}$ is an equilibrium or a limit cycle, the
stable foliation theorem asserts that a neighborhood $\mathcal{N}$
of $\mathcal{A}$ can be taken as a ball, and consequently $W(x)$ can
be taken as $\mathrm{dist}(x, \mathcal{A})^{2}$.
\end{rem}

\medskip

 When noises are added to the ODE system \eqref{ODE1},
our theory requires characterizations and estimates of stochastic
quantities such as mean square displacement and entropy-dimension
formula of stationary measures of the Fokker-Planck equation
  \eqref{FPE1} associated with the SDE system \eqref{SDE1}. It turns
  out that, for these quantities to be well-defined, the following
  condition on the stationary measures of  \eqref{FPE1} is needed:
\medskip

\begin{itemize}
  \item[{\bf H$^1$)}] For each $\epsilon \in (0, \epsilon^{*})$, the Fokker-Planck equation
  \eqref{FPE1} admits a unique stationary measure
    $\mu_{\epsilon}$ such that
$$
 \lim_{\epsilon \rightarrow 0} \frac{\mu_{\epsilon}( \mathbb{R}^{n}
   \setminus \mathcal{N})}{\epsilon^{2}}  = 0,
$$
and moreover, there are constants $p, R_0>0$ such that
$$
  \mu_{\epsilon}( \{ x \,:\, |x| > r \}) \leq e^{-\frac{r^{p}}{\epsilon^{2}}}
$$
for all $r>R_0$ and all $\epsilon \in (0, \epsilon^{*})$.
\end{itemize}
\medskip

Throughout the rest of the paper, for any fixed $\epsilon\in
(0,\epsilon^*)$, we let $\mu_{\epsilon}$ denote the unique
stationary probability measure of \eqref{SDE1} or the stationary
measure of  \eqref{FPE1} and let $u_\epsilon(x)$ or, when it does
not cause confusion, $u(x)$ stand for the (classical) stationary
solution of equation \eqref{ssFPE}, which is the density function of
$\mu_{\epsilon}$.

To estimate these stochastic quantities mentioned above
  rigorously, it is essential to perform estimates on the
concentration of $\mu_{\epsilon}$ both near and away from $\mathcal{A}$.
 In Section 3.1, we will conduct estimates on the local
concentration of $\mu_\epsilon$ in the vicinity of $\mathcal{A}$ by
making use of assumption {\bf H$^0$)} and give estimates of the tails
of $\mu_\epsilon$ by providing a sufficient condition which ensures
the validity of the condition {\bf H$^1$)}.

 We remark that the
estimation in Section 3.1 only provides one of many approaches to verify {\bf
  H$^{0}$)} and {\bf H$^{1}$)}. Essentially {\bf
  H$^{1}$)} assumes that $\mu_{\epsilon}$ has sufficient concentration on an isolating
neighborhood such that we can
focus on the local analysis in the vicinity of the global
attractor. As discussed in Example \ref{ex2} below, such concentration is satisfied by many problems in
applications, although it may be
difficult to give generic sufficient conditions. In particular, the
quasi-potential function defined in \eqref{QPP}, if exists, leads to
the desired concentration of $\mu_{\epsilon}$ immediately. If a
quasi-potential function as in \eqref{QPP} exists and is
differentiable, then it is a Lyapunov function of \eqref{ODE1} \cite{FW}. This
provides an alternative way of verifying {\bf H$^{0}$)}.  For
results regarding high regularity of the quasi-potential function, see
\cite{day1994regularity, day1985some}.

\medskip

\subsection{Estimating tails of stationary measures}

The purpose of this subsection is to provide an
  alternative way to verify assumption {\bf H$^{0}$)} and {\bf
    H$^{1}$)}. This is important in applications as rigorously verifying the
  quasi-potential landscape may be difficult for some
  models. Instead of using the quasi-potential function, we use a
  suitable Lyapunov function of  system \eqref{ODE1} to facilitate our study. To
  characterize the property of the desired Lyapunov function, the following
  definitions are necessary.

A compact function $U$ on a connected open set $S\subset\R^n$ is
said to be of the {\it
  class $\mathcal{B}^{*}$} in $S$ if  there is a
constant $p>0$ and a function $H(\rho) \in L^{1}_{loc}([\rho_{0},
\rho_M))$   such that
$$
 H(\rho) \geq |\nabla U(x)|^{2},\quad x \in \Gamma_{\rho}(U)
$$
and
$$
  \int_{\rho_0}^{\rho}\frac{1}{H(s)} \mathrm{d}s \geq  |x|^{p}, \quad x
 \in \Gamma_{\rho}(U)
$$
for all $\rho\in (\rho_0,\rho_M)$,  where  $\rho_0=\inf_{x\in S}
(x)$.


\begin{rem}
According to the definition, a compact function $U(x)$ is of class
$\mathcal{B}^{*}$ in $\mathbb R^n$ if (i) $U(x)$ has bounded first
order derivative and (ii) $\lim_{|x| \rightarrow \infty}
\frac{U(x)}{|x|^{p}} > 0$ for some $p > 0$. We will show that when
\eqref{FPE1} admits a class $\mathcal{B}^{*}$ Lyapunov function, its
stationary measure has an exponential tail. One example of class
$\mathcal{B}^{*}$ function will be given at the end of this
subsection.
\end{rem}

 We will estimate the tails of
stationary measures of \eqref{FPE1} by dividing $\R^n\setminus \cal
N$ into two regions: a neighborhood ${\cal N}_\infty$ of $\infty$ in
$\R^n$,  i.e., the complement of a
  sufficiently large compact set, and the intermediate region ${\cal N}_*$ between ${\cal
N}_\infty$ and $\cal N$. We make the following hypothesis:
\medskip
\begin{itemize}
\item[{\bf H$^{2}$)}] There is a positive function $U\in C^2(\R^n\setminus \cal
A)$ satisfying the following properties:
\begin{itemize} \item[{\rm
i)}] $\lim_{x\to\infty} U(x)=\infty$; \item[{\rm ii)}] There exists
a constant $\rho_m>0$  such that $U$ is a uniform Lyapunov function
of the family \eqref{FPE1} of class ${\cal B}^*$ in ${\cal
N}_\infty=:\mathbb{R}^{n}\setminus \Omega_{\rho_{m}}(U)$;
\item[{\rm iii)}] There exists a constant $\bar\rho_m\in (0,\rho_m)$ such
that $U$ is a uniform weak Lyapunov function of the family
\eqref{FPE1} in ${\cal N}_*=:\R^n\setminus {\cal N}_\infty\setminus
\Omega_{\bar\rho_{m}}(U) = \Omega_{\rho_{m}}(U) \setminus \Omega_{\bar{\rho}_{m}}(U)$;

\item[{\rm  iv)}] $\nabla U(x)\ne 0$,
$x\in \R^n\setminus  \Omega_{\bar\rho_{m}}(U)$;
\item[{\rm v)}] $\Omega_{\bar\rho_m} (U)\subset {\cal N}$.
\end{itemize}
\end{itemize}
\medskip


\begin{rem} 1) We note that when {\bf H$^{2}$)} holds, Theorem ~\ref{exist1} asserts the
existence of a unique stationary measure of \eqref{FPE1} for each
$\epsilon\in (0,\epsilon_*)$.

2) With the hypotheses {\bf H$^0$)}, {\bf H$^{2}$)},  the ODE system \eqref{ODE1} is
 dissipative in ${\cal N}_\infty$, strongly dissipative in $\cal
N$, and remains dissipative in ${\cal N}_*$ but with small
dissipation rate proportional to $\epsilon^2$.

3) The purpose of introducing {\bf H$^{2}$)} is to give a Lyapunov
function-based sufficient condition of {\bf H$^{0}$)} and {\bf
  H$^{1}$)}. Except in this subsection, our estimates are based only on {\bf H$^{0}$)} and {\bf
  H$^{1}$)}.

\end{rem}




\medskip

 We first estimate the concentration of  stationary
measures of \eqref{FPE1} in the region ${\cal N}_\infty$,
  which verifies the second part of {\bf H$^{1}$)}.
\medskip

\begin{pro}
\label{faraway}  If {\bf H$^{2}$)} holds, then there exist
positive constants $\beta, R_0, p$ such that
\[
\mu_{\epsilon}(\mathbb{R}^{n}\setminus B(0,r))\leq e^{-\beta
\frac{r^{p}}{\epsilon^{2}}},\qquad \epsilon \in (0, \epsilon_{*}) \,,
\]
for  all $r \ge R_{0}$.

\end{pro}

\begin{proof} Let $U$ be the uniform Lyapunov function of \eqref{FPE1}
for $\epsilon\in (0,\epsilon_*)$, according to {\bf H$^{2}$)}. Since $U$ is of
class ${\cal B}^*$ in ${\cal N}_\infty$,  there is a function
$H(\rho) \in L^{1}_{loc}([\rho_{m}, \infty))$, where $\rho_m$
denotes the essential lower bound of $U$, such that
$$
  H(\rho) \geq |\nabla U(x)|^{2},\quad x \in \Gamma_{\rho}(U)
$$
and
$$
  \int_{\rho_m}^{\rho}\frac{1}{H(s)} \mathrm{d}s \geq  |x|^{p}, \quad x
  \in \Gamma_{\rho}(U)
$$
for all $\rho > \rho_m$. Using positive definiteness of $A(x)$, we
let $C_{0} > 0$ be a constant such that
$$
 \epsilon^{2} H(\rho) \geq C_{0} \frac{1}{2}\epsilon^{2}
 \sum_{i,j=1}^{n}a_{ij}\partial_{i}U(x) \partial_{j}U(x)
$$
and denote $H_{1}(\rho) = \epsilon^{2}H(\rho)/C_{0}$.  It follows
from Lemma ~\ref{lem41} that
\[
\mu_{\epsilon}(\mathbb{R}^{n}\setminus \Omega_{\rho}(U)) \leq
e^{-\gamma\int_{\rho_m}^{\rho}\frac{1}{H_{1}(s)} \mathrm{d}s}\leq
e^{-\frac{\gamma
    C_{0}}{\epsilon^{2}}\int_{\rho_m}^{\rho}\frac{1}{H(s)} \mathrm{d}s}\leq e^{-\frac{\gamma C_{0}}{\epsilon^{2}}|x|^{p} }
\]
for each $x \in \Gamma_{\rho}(U)$ whenever $\rho>\rho_m$, where
$\gamma>0$ is the Lyapunov constant of $U$.

Let $R_{0} = \max_{x\in
  \Gamma_{r_m}(U)} |x|$, and for each $r \ge R_{0}$,  denote $\rho(r)
= \min_{|y| = r}U(y)$. Let $r\ge R_0$ and take $x\in
\Omega_{\rho(r)}\cap B(0,r)$. Then $\Omega_{\rho(r)}\subset B(0,r)$
and
\[
\mu_{\epsilon}(\mathbb{R}^{n}\setminus B(0,r))\le
\mu_{\epsilon}(\mathbb{R}^{n}\setminus \Omega_{\rho(r)}(U))\le
e^{-\frac{\gamma
C_{0}}{\epsilon^{2}}|x|^{p}}=e^{-\frac{\beta}{\epsilon^{2}}r^{p}},
\]
where $\beta = \tau C_{0}$.
\end{proof}


Next, we estimate the  concentration of  stationary probability
measures of \eqref{FPE1} in the intermediate  region ${\cal N}_*$ to
verify the first part of {\bf H$^{1}$)} . For any connected open set
$S\subset \R^n$, we note that
$\mu_\epsilon|_S=:\mu_\epsilon/\mu_\epsilon(S)$ is a stationary probability
measure of \eqref{FPE1} on $S$.
\medskip

\begin{lem}\label{lem31} Assume {\bf H$^0$)} and  let  $W$ be  as in {\bf H$^0$)}
and $\rho_0>0$ be such that $\Omega_{\rho_0}(W)\subset \cal N$ and
$\Gamma_{\rho_0}\cap \cal A\ne \emptyset$. Then there is an
$\epsilon_0\in (0,\epsilon_*)$ such that
\[
  \mu_{\epsilon}(\mathcal{N} \setminus
\Omega_{\rho_0}(W)) \leq e^{-\frac{C_{1}}{\epsilon^{2}}},\qquad \epsilon\in
(0,\epsilon_0),
\]
where $C_1>0$ is a constant independent of $\epsilon$.
\end{lem}

\begin{proof} It is easy to see that there exists an $\epsilon_0\in
(0,\epsilon_*)$ such that $W$ becomes a uniform Lyapunov function
of \eqref{FPE1} for all $\epsilon\in (0,\epsilon_0)$ in $S=:\cal N$,
with upper bound $\rho^0=\sup \{\rho>0: \Omega_\rho(W)\subset\cal
N\}$ and essential lower bound $\rho_0$. The lemma now follows from
an application of Lemma ~\ref{lem41} to $\mu_\epsilon|_{\cal N}$ with
\[
H(\rho)=\frac{1}{2}\epsilon^{2}\min_{x\in
\Gamma_\rho(W)}(\sum_{i,j=1}^{n}a_{ij}(x)\partial_{x_{i}}W(x)\partial_{x_{j}}W(x)),\qquad
\rho\in (\rho_0,\rho^0).
\]
\end{proof}

\medskip

\begin{pro}
\label{mid} Assume {\bf H$^0$)}, {\bf H$^{2}$)} and  let $W$
be as in {\bf H$^0$)} and $R_0$ be as in {\rm
Proposition~\ref{faraway}}. Then there are constants $\rho_0, c_0>0$
and $\epsilon_0\in (0,\epsilon_*)$ such that
$$
\mu_{\epsilon}(B(0, R_{0}) \setminus \Omega_{\rho_0}(W)) \leq
e^{-\frac{c_{0}}{\epsilon^{2}}},\qquad \epsilon \in (0, \epsilon_{0}) \,.
$$
\end{pro}

\begin{proof} Let $U,\rho_m,\bar\rho_m$ be as in {\bf H$^{2}$)} and
 $
\bar\rho_M\in (\rho_m,\infty)$  be such that $B(0,R_0)\subset
\Omega_{\bar\rho_M} (U)$. Without loss of generality, we assume that
$\Gamma_{\bar\rho_m}(U)\cap \cal A=\emptyset$.
 Denote
$\rho^0=\sup\{\rho>0: \Omega_\rho(W)\subset\cal N\}$ and let
$\rho_0\in (0,\rho^0)$, $\bar\rho_*\in (\rho_0,\rho^0)$  be such
that $\Gamma_{\bar\rho_0}(W)\cap \cal A=\emptyset$,
$\Omega_{\rho_0}(W)\subset\Omega_{\bar\rho_m} (U)\subset
\Omega_{\bar\rho_*}(U)\subset \Omega_{\rho^0}(W)$.

By Lemma~\ref{lem31}, there exists an $\epsilon_0>0$ and a constant
$C_1>0$ such that
\begin{equation}\label{e1}
  \mu_{\epsilon}(\mathcal{N} \setminus
\Omega_{\rho_0}(W)) \leq e^{-\frac{C_{1}}{\epsilon^{2}}},\qquad \epsilon\in
(0,\epsilon_0).
\end{equation}

Since $U$ is a uniform Lyapunov function of \eqref{FPE1} for
$\epsilon \in (0,\epsilon_*)$ on $S=:\Omega_{\bar\rho_M}(U)$, an
application of Lemma~\ref{lem33} to $\mu_\epsilon|_S$ with
\begin{eqnarray*}
 && H_{1}(\rho)=\frac{1}{2}\epsilon^{2}\min_{x\in \Gamma_\rho(U)}(\sum_{i,j = 1}^{n}a_{ij}(x)\partial_{i}
  U(x)\partial_{j} U(x)),\\
&& H_{2}(\rho)=\frac{1}{2}\epsilon^{2}\max_{x\in
\Gamma_\rho(U)}(\sum_{i,j = 1}^{n}a_{ij}(x)\partial_{i}
  U(x)\partial_{j} U(x)),\quad \rho\in (\bar\rho_m, \bar\rho_M),
  \end{eqnarray*}
  yields that there is a constant $C_2>0$ independent of $\epsilon$
  such that
  \begin{eqnarray}
  \mu_\epsilon(B(0,R_0)\setminus \Omega_{\rho_0}(W))
  &\le& \mu_\epsilon(S\setminus \Omega_{\bar\rho_m}(U))+\mu_\epsilon(\cal N\setminus
  \Omega_{\rho_0}(W))\nonumber\\
&\le& C_2\mu_\epsilon(\Omega_{\bar\rho_*}(U)\setminus
\Omega_{\bar\rho_m}(U))+\mu_\epsilon(\cal N\setminus
\Omega_{\rho_0}(W))\nonumber\\
& \le& (C_2+1) \mu_\epsilon(\cal N\setminus
\Omega_{\rho_0}(W)).\label{e2}
\end{eqnarray}
The proposition now easily follows from \eqref{e1}, \eqref{e2}.

\end{proof}

 Now,  Propositions~\ref{faraway},~\ref{mid} immediately
yields the following result.
\medskip
\begin{cor}
\label{cor31}
Conditions {\bf H$^0$), H$^{2}$)} imply {\bf H$^1$)}.
\end{cor}

\medskip
 Below, we give a simple example that satisfies {\bf H$^{0}$)}
and {\bf H$^{2}$)}.

\begin{ex}
 Consider
\begin{equation}
  \label{limitcycle}
 \left\{\begin{array}{l}
  x' = y + x(1 - x^{2} - y^{2}) \\
 y' = -x + y(1 - x^{2} - y^{2}).
\end{array}\right.
\end{equation}
Let
$$
  U(x,y) = \sqrt{x^{2} + y^{2} } h( \sqrt{x^{2} + y^{2}}) + (x^{2} + y^{2}
  - 1)^{2}(1 - h(\sqrt{x^{2} + y^{2}}))
$$
where $h(r)$ is a nonnegative nondecreasing $C^{2}$ cut-off function such that $h(r) = 0$ for
$r \leq 1.3$ and $h(r) = 1$ for $r \geq 1.4$. Then it
is easy to verify that $U(x,y)$ is of class $\mathcal{B}^{*}$ in
$\{(x,y) \, : \, \sqrt{x^{2} + y^{2}} \geq 1.4\}
$. When $\epsilon$ is sufficiently small, the other conditions in {\bf
  H$^{2}$)} and {\bf H$^{0}$)} are also satisfied.
\end{ex}

\medskip

\begin{rem}
Example 3.1 is not a biological example. The purpose of having this simple
example is to show that two Lyapunov functions can be ``glued'' to
verify {\bf H$^{2}$)}. In applications, if an ODE system has
inward-pointing vector field far away from the origin, it is usually
easy to find a Lyapunov function of class $\mathcal{B}^{*}$ in
$\mathcal{N}_{\infty}$. This Lyapunov function in
$\mathcal{N}_{\infty}$ may not have the Lyapunov property near the
attractor. On the other hand, many
ODE systems in biological models, such as mass-action systems,
admit natural Lyapunov functions \cite{horn1972general,
  gopalkrishnan2013lyapunov, feinberg1987chemical,
  feinberg1995existence}, which are not of class $\mathcal{B}^{*}$ in $\mathcal{N}_{\infty}$. Often two Lyapunov functions can be ``glued'' together
to satisfy {\bf H$^{2}$)}, which, by Corollary \ref{cor31}, rigorously leads to the desired
concentration of $\mu_{\epsilon}$ needed in the rest of this
paper. We remind readers that there are some systematic ways to propagate
``local'' Lyapunov functions to construct a global Lyapunov function
\cite{athreya2012propagating, herzog2014noise}, which can be used to check the validity of {\bf H$^{2}$)} in
applications.

\end{rem}

\medskip

\begin{ex}[Toggle switch model]
\label{ex2}

Consider a gene circuit that consists of two genes $G_{A}$ and
$G_{B}$. $G_{A}$ and $G_{B}$ produces proteins $A$ and $B$, respectively. Assume that protein $A$ can
turn off gene $G_{B}$ by binding with its promotor, and vice
versa. Once turned off, each gene turns back on at a certain rate due to the
unbinding of the protein. Let $x$ and $y$ be the concentration of $A$
and $B$, respectively. Then the time evolution of concentration can be
described by the following ordinary differential equation
\begin{eqnarray}
\label{switch}
 \frac{\mathrm{d}x}{\mathrm{d}t}& = & \frac{1}{1 + \frac{y^{2}}{b +
     x^{2}}} -  x \\\nonumber
\frac{\mathrm{d}y}{\mathrm{d}t}& = & \frac{1}{1 + \frac{x^{2}}{b +
     y^{2}}} - y \,,
\end{eqnarray}
where $b$ is a constant \cite{newby2012isolating}. It
is easy to see that equation \eqref{switch} admits two stable
equilibria $A$ and $B$ and one saddle
equilibrium $C$ (See Figure \ref{toggleswitch}).
When adding additive white noise perturbation
$\epsilon \dot W$ to both equations of \eqref{switch}, it is known
that there exists a quasi-potential function $V(x,y)$ such that
 the corresponding stationary measure $\mu_{\epsilon}$
has the WKB approximation \eqref{wkb}. Therefore, if
we let $\mathcal{N}$ be the neighborhood of $\{A, B\}$, then
 $\mu_{\epsilon}$ satisfies the concentration condition
{\bf H$^{1})$}. The quasi-potential landscape of equation
\eqref{switch} is well-studied. We refer the reader
to \cite{newby2014asymptotic} for  a numerical
computation of $V(x, y)$  (see also
\cite{lu2014construction} for a similar gene switch model and its
numerical quasi-potential landscape).

\begin{figure}[h]
\centerline{\includegraphics[width = 9cm]{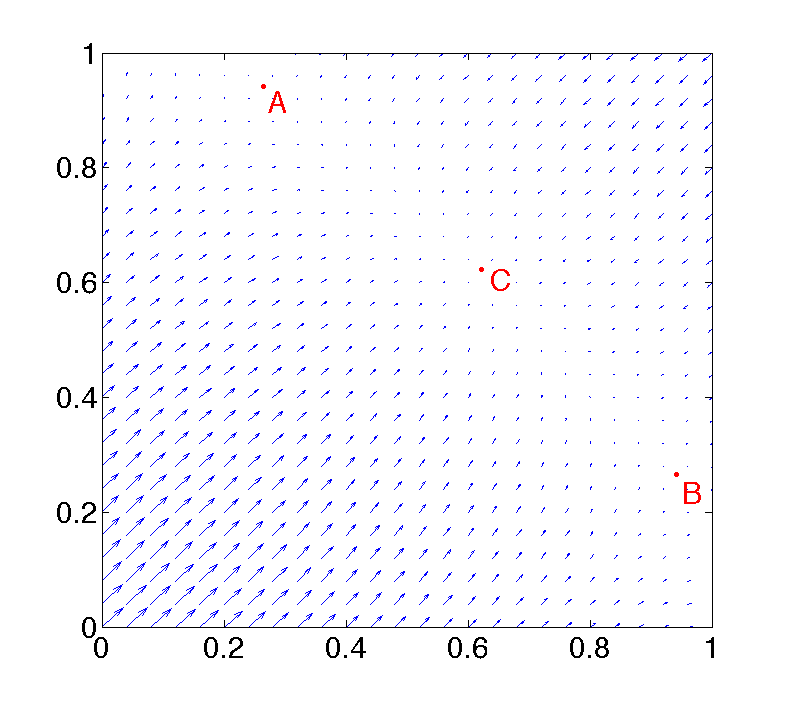}}
\caption{\label{toggleswitch} Vector field of Equation \eqref{switch}
  with $b = 0.25$. $A$ and $B$ are stable equilibriums and $C$ is a
  saddle point.}
\end{figure}

\end{ex}

\begin{rem}
 Besides the toggle switch model in Example \ref{ex2},
  many other biological models, including biochemical oscillation systems
\cite{wang2008potential}, genetic circuits \cite{wang2010potential},
gene regulatory networks \cite{huang2013escape}, and cell cycle network
\cite{lv2015energy}, are known to have similar
quasi-potential landscape, hence {\bf H$^{1}$)} is
satisfied by all these systems. Also see
\cite{zhou2012quasi, ao2007existence} for theoretical studies of
quasi-potential functions in biological systems and
\cite{zhou2008adaptive, zhou2012quasi, ren2004minimum, heymann2008geometric} for
numerical computation methods of the quasi-potential function.

\end{rem}

\subsection{Concentration of stationary measures near the global attractor}
 Let $L_1$ be as in {\bf H$^0$)} and denote $\gamma_{0}$ as
the Lyapunov constant of $W$. The following lemma is straightforward
from the $C^{2}$ smoothness of $W$.
\medskip

\begin{lem}\label{conditions} Assume {\bf H$^0$)}. Then there are
positive constants $\kappa, L_{1}, L_{2}, K_{1}, K_{2}$
 such that
\begin{eqnarray*}
  && -\kappa|\nabla W(x)|^{2} \leq \nabla W(x)\cdot f(x) \leq -\gamma_{0}
   |\nabla W(x)|^{2},\\
&&
   L_{1}\mathrm{dist}^{2}(x,\mathcal{A}) \leq W(x) \leq
   L_{2}\mathrm{dist}^{2}(x,\mathcal{A}),\\
&&
   K_{1} \mathrm{dist}(x,\mathcal{A}) \leq |\nabla W (x)| \leq K_{2}
   \mathrm{dist} (x,\mathcal{A}),
\end{eqnarray*}
$x \in \mathcal{N}$.
\end{lem}
\medskip

Below, for any bounded set $A\subset \mathbb{R}^{n}$ and $r>0$, we
denote $B(A,r) := \{x\in \mathbb{R}^{n} \, :\, \mathrm{dist}(x,A)
\leq r \}$ as the $r$-neighborhood of $A$. The following theorems
give new estimations in the vicinity of $\mathcal{A}$.
\medskip

\begin{thm}
\label{accurate} If both {\bf H$^0$)} and {\bf H$^1$) }
hold, then   for any $0<\delta \ll 1$ there exist constants
$\epsilon_{0},M>0$  such that
\begin{displaymath}
   \mu_{\epsilon}(B( \mathcal{A},M\epsilon)) \geq 1-\delta,
\end{displaymath}
whenever $\epsilon\in (0,\epsilon_{0})$.
\end{thm}

\begin{proof}  Fix a $\rho_0>0$ such that $\Omega_{\rho_0}(W)\subset \cal N$ and
$\Gamma_{\rho_0}\cap \cal A =  \emptyset$. Then by Lemma ~\ref{lem31}
there are constants $\epsilon_0,C_1>0$ such that $\mu_{\epsilon}(
\mathcal{N}\setminus \Omega_{\rho_0}(W)) < e^{-C_{1}/\epsilon^{2}}$,
$\epsilon\in (0,\epsilon_0)$. Since by {\bf H$^1$)},   $\mu_{\epsilon}(
\mathbb{R}^{n} \setminus \mathcal{N})  = o(\epsilon^{2})$,
$0<\epsilon\ll 1$,  we only need to estimate
$\mu_\epsilon(\Omega_{\rho_0}(W))$.

\medskip

For any given $0<\tilde\rho_{0}<\rho_0$ and any $0<\Delta \rho \ll
1$, consider the following  $C^{2}$ cut-off function
\begin{displaymath}
   \phi(\rho) = \left\{\begin{array}{ll}
0,&\rho\le\tilde\rho_{0}\\
\frac{3}{\Delta\rho^{4}}(\rho - \tilde\rho_{0})^{5} -
\frac{8}{\Delta\rho^{3}}(\rho - \tilde\rho_{0})^{4}+
\frac{6}{\Delta\rho^{2}}(\rho - \tilde\rho_{0})^{3},&\tilde\rho_{0}<\rho<\tilde\rho_{0}+\Delta\rho\\
\rho - \tilde\rho_{0},&\rho\ge\tilde\rho_{0}+\Delta\rho.
\end{array}\right .
\end{displaymath}

Let $u=u_\epsilon $ be a density function of
$\mu_\epsilon$. It follows from Theorem ~\ref{lem23} that
\begin{eqnarray}
 &&\int_{\Omega_{\rho_0}(W)}\phi'(W(x))\left (\sum_{i,j=1}^{n} \frac{1}{2}\epsilon^{2}a_{ij}(x)\partial^{2}_{ij}W(x)
   +
   \sum_{i=1}^{n}f_{i}(x)\partial_{i}W(x)\right
 )u(x)\mathrm{d}x \\
&&\quad +
   \int_{\Omega_{\rho_0}(W)}\phi''(W(x))\left (\sum_{i,j=1}^{n}\frac{1}{2}\epsilon^{2}a_{ij}(x)
   \partial_{i}W(x)\partial_{j}W(x)\right )u(x)\mathrm{d}x\nonumber\\
   &=&\int_{\Gamma_{\rho_0}(W)}\left (\sum_{i=1}^{n}\sum_{j=1}^{n}\frac{1}{2}\epsilon^{2}a_{ij}(x)\partial_{i}W(x)
   \nu_{j}(x)\right )u(x)\mathrm{d}x \geq 0, \label{eq3-1}
\end{eqnarray}
 where  $\nu_{j} = \frac{\partial
  W_{j}}{|\nabla W|}$ for each $j$.

To estimate the first term on the left hand side of \eqref{eq3-1},
we note by  definition of $\phi(\rho)$ that
\begin{eqnarray*}
  && \int_{\Omega_{\rho_0}(W)}\phi'(W(x))\left (\sum_{i,j=1}^{n}\frac{1}{2}\epsilon^{2}a_{ij}(x)\partial^{2}_{ij}W(x)
   +
   \sum_{i=1}^{n}f_{i}(x)\partial_{i}W(x)\right
 )u(x)\mathrm{d}x \\
&=&
   \int_{\Omega_{\rho_0}(W)\backslash
     \Omega_{\tilde\rho_{0}(W)}}\phi'(W(x))\left (\sum_{i,j=1}^{n}\frac{1}{2}\epsilon^{2}a_{ij}(x)\partial^{2}_{ij}W(x)
   +
   \sum_{i=1}^{n}f_{i}(x)\partial_{i}W(x)\right
 )u(x)\mathrm{d}x.
\end{eqnarray*}

 Denote
\begin{eqnarray*}
   \bar{\sigma} &=& n^{2}\max_{1\le i,j\le n, x \in U_{\rho_0(W)}}
   |a_{ij}(x)|,\\
  D &=& \max_{1\le i,j\le n, x \in U_{\rho_0}(W)} \partial^{2}_{ij}W(x)
\end{eqnarray*}
and let $M_{1} = D\bar{\sigma}/\gamma_{0} K^{2}_{1}$,  where
$\gamma_{0}$ and $K_{1}$ are constants in Lemma~\ref{conditions}.
Let $\epsilon_{1}>0 $ be such that
$B(\mathcal{A},\sqrt{M_{1}}\epsilon) \subset \Omega_{\rho_0}(W)$ for
all $0<\epsilon<\epsilon_{1}$. Then
\begin{displaymath}
    \sum_{i,j=1}^{n}\frac{1}{2}\epsilon^{2}a_{ij}(x)\partial^{2}_{ij}W(x) \leq \frac{\gamma_{0}}2 |\nabla
    W(x)|^2,
\end{displaymath}
for all $x$ with $\mathrm{dist}(x,\mathcal{A})>\sqrt{M_{1}}\epsilon$
and $\epsilon\in (0,\epsilon_1)$. It follows from the property of
strong Lyapunov function that
\[
   \sum_{i,j=1}^{n}\frac{1}{2}\epsilon^{2}a_{ij}(x)\partial^{2}_{ij}W(x)  +
   \sum_{i=1}^{n}f_{i}(x)\partial_{i}W(x) \leq -\frac{1}{2}\gamma_{0} |\nabla W(x)|^{2}
\]
for all $x$ with $\mathrm{dist}(x,\mathcal{A})>\sqrt{M_{1}}\epsilon$
and $\epsilon\in (0,\epsilon_1)$. Let ${\cal
D}_0=\Omega_{\rho_{0}}(W)\backslash
\Omega_{\tilde\rho_{0}+\Delta\rho}$ and $\cal
D=\Omega_{\tilde\rho_{0}+\Delta\rho}(W)\backslash
\Omega_{\tilde\rho_{0}}$.  By Lemma~\ref{conditions}, we also have
\begin{eqnarray*}
   \min_{{\cal D}_0} |\nabla W(x)|^{2} &\geq&  K_{1}^2 \min_{{\cal D}_0}
   \mathrm{dist}^2(x,\mathcal{A}) \ge \frac{K_1^2}{L_2}\min_{{\cal D}_0} W(x)
   =\frac{K_1^2}{L_2}(\tilde\rho_0+\Delta\rho)\\\nonumber
&\ge& \frac{K_1^2}{L_2}\max_D W(x)
   \ge \frac{K_1^2L_1}{L_2}\max_{\cal D}\mathrm{dist}^2 (x,\mathcal{A})\\\nonumber
   &\ge& \frac{K_1^2L_1}{K_2^2L_2}
    \max_{\cal D} |\nabla W(x)|^{2}=: C_1 \max_{\cal D} |\nabla
    W(x)|^{2},
    \end{eqnarray*}
where $K_{1}, K_{2}, L_{1}, L_{2}$ are as in Lemma~\ref{conditions}.
Therefore,
\begin{eqnarray}
&&\int_{\Omega_{\rho_0}(W)\backslash
\Omega_{\tilde\rho_{0}}(W)}\phi'(W(x))\left (\sum_{i,j=1}^{n}\frac{1}{2}\epsilon^{2}a_{ij}(x)\partial^{2}_{ij}W(x)
+
\sum_{i=1}^{n}f_{i}(x)\partial_{i}W(x)\right
)u(x)\mathrm{d}x \\\nonumber
&\leq & -\gamma_{0}\int_{\Omega_{\rho_0}(W)\backslash
  \Omega_{\tilde\rho_{0}}(W)}\phi'(W(x))|\nabla
W(x)|^{2}u(x)\mathrm{d}x\\\nonumber
& \leq& -\gamma_{0} \int_{\Omega_{\rho_0}(W)\backslash
   \Omega_{\tilde\rho_{0}+\Delta\rho}(W)} |\nabla W(x)|^{2} u(x) \mathrm{d}x\nonumber\\
&\le&
 -\gamma_{0}\min_{{\cal D}_0} |\nabla W(x)|^{2}\int_{\Omega_{\rho_0}(W)\backslash
   \Omega_{\tilde\rho_{0}+\Delta\rho}(W)} u(x) \mathrm{d}x\nonumber\\
 &\leq&-\gamma_{0} C_{1}\max_{\cal D} |\nabla W(x)|^{2}\mu_\epsilon(\Omega_{\rho_0}(W)\backslash
   \Omega_{\tilde\rho_{0}+\Delta\rho}(W)). \nonumber\label{c1}
\end{eqnarray}

Note that $|\phi''(x)| \leq 4$. The second term on the left hand side of \eqref{eq3-1} simply
satisfies the following estimate:
\begin{eqnarray}
&&\int_{\Omega_{\rho_0}(W)}\phi''(W(x))\sum_{i,j=1}^{n}\frac{1}{2}\epsilon^{2}a_{ij}(x)\partial_{i}W(x)\partial_{j}W(x)
u\mathrm{d}x\nonumber\\
& = & \int_{\Omega_{\tilde\rho_{0}+\Delta\rho}(W)\backslash
\Omega_{\tilde\rho_{0}}(W)}
\phi''(W(x))\sum_{i,j=1}^{n}\frac{1}{2}\epsilon^{2}a_{ij}(x)\partial_{i}W(x)\partial_{j}W(x)
u
\mathrm{d}x\nonumber\\
&\leq &\frac{2}{\Delta\rho}\epsilon^{2}\bar{\lambda}\max_{\cal D}
{|\nabla
W(x)|^{2}\mu_\epsilon(\Omega_{\tilde\rho_{0}+\Delta\rho}\backslash
\Omega_{\tilde\rho_{0}}}),\label{c2}
\end{eqnarray}
where $\bar{\lambda}=\sup_{x \in \Omega_{\rho_0}(W)}\lambda_M(x)$
with $\lambda_M(x)$ being the largest eigenvalue of matrix $A(x)$
for each $x \in \Omega_{\rho_0}(W)$.

It now follows from \eqref{eq3-1}-\eqref{c2} that
\begin{equation}\label{c3}
   -\gamma_{0} C_{1} \mu(\Omega_{\rho_0}(W)\backslash
   \Omega_{\tilde\rho_{0}+\Delta\rho}(W)) +
   \frac{2}{\Delta\rho}\epsilon^{2}\bar{\lambda}\mu(\Omega_{\tilde\rho_{0}+\Delta\rho}(W)\backslash
   \Omega_{\tilde\rho_{0}}(W))\geq 0.
\end{equation}

Let $\rho_{1} = L_{2}M_{1}\epsilon^{2}$.  We have by
Lemma \ref{conditions} that
$B(\mathcal{A},\sqrt{M_{1}}\epsilon)\subset \Omega_{\rho_{1}}$.
Consider function $F(\rho) =
\mu_{\epsilon}(\Omega_{\rho_0}(W)\backslash \Omega_{\rho}(W))$,
$\rho\in [\rho_{1}, \rho_0]$. Since $\tilde\rho_0$ is arbitrary,
\eqref{c3} with $\rho$ in place of $\tilde\rho_0$ becomes
\[
   -\gamma_{0} C_{1}F(\rho+\Delta\rho) +
   \frac{2}{\Delta\rho}\epsilon^{2}\bar{\lambda}(F(\rho) - F(\rho + \Delta\rho))
   \geq 0,\qquad  \rho\in [\rho_{1}, \rho_0].
\]
Taking  limit $\Delta\rho \rightarrow 0$ in the above yields
\begin{displaymath}
   \gamma_{0} C_{1}F(\rho) + 2\epsilon^{2}\bar{\lambda}F'(\rho)\le 0.
\end{displaymath}
Hence, by  Gronwall's inequality, we have
\begin{equation}\label{c4}
   F(\rho) \leq F(\rho_{1})e^{-\frac{\gamma_{0}
       C_{1}}{2\epsilon^{2}\bar{\lambda}}(\rho-\rho_{1})}, \quad \rho \in
   [\rho_{1}, \rho_0]\,.
\end{equation}

For a given sufficiently small $\delta>0$, we let
\begin{displaymath}
   M_{2} = L_{2}M_{1} - \frac{2\bar{\lambda}}{\gamma_{0}} C_{1}\log \frac{\delta}{2} \,.
\end{displaymath}
Then it is easy to see from \eqref{c4} that
\begin{equation}\label{c5}
   F(M_{2}\epsilon^{2}) = \mu_{\epsilon}(\Omega_{\rho_0}(W)) - \mu_{\epsilon}(\Omega_{M_{2}\epsilon^{2}}(W))
   \leq \frac{\delta}{2},
\end{equation}
whenever $\epsilon<\epsilon_{2} =: \frac{\rho_{0}}{M_{2}}$.

 Since $\mu_{\epsilon}(\mathbb{R}^{n}\setminus
\Omega_{\rho_0}(W)) < e^{-C_{0}/\epsilon^{2}} + o( \epsilon^{2})$,
there exists an $\epsilon_{3} > 0$ such that
$\mu_{\epsilon}(\mathbb{R}^{n}\setminus \Omega_{\rho_0}(W)) <
\frac{\delta}{2}$,  i.e., $\mu_{\epsilon}(\Omega_{\rho_0}(W))
>1- \frac{\delta}{2}$, for all $\epsilon \in (0, \epsilon_{3})$.
Hence by \eqref{c5},
\begin{displaymath}
   \mu_{\epsilon}(\Omega_{M_{2}\epsilon^{3}}(W)) \geq 1-\delta,\quad 0<\epsilon<\epsilon_0,
   \end{displaymath}
where $\epsilon_{0} = \min \{\epsilon_{1},\epsilon_{2},\epsilon_{3}
\}$.

 Let $M = \sqrt{\frac{M_{2}}{L_{1}}}$.  Then by
 Lemma~\ref{conditions},
\begin{displaymath}
   \Omega_{M_{2}\epsilon^{2}}(W)\subset B( \mathcal{A},M\epsilon),
\end{displaymath}
and therefore,
\begin{displaymath}
   \mu_{\epsilon}(B( \mathcal{A}, M\epsilon))\geq \mu_{\epsilon}(\Omega_{M_{2}\epsilon^{2}}(W)) \geq
   1-\delta,\quad 0<\epsilon<\epsilon_{0}.
\end{displaymath}
This completes the proof.
\end{proof}
\medskip

\begin{rem}
1) From the proof of Theorem ~\ref{accurate},  one sees that the
constant $M$ grows in a logarithm rate as $\delta$ decreases. In
fact, for a fixed small $\epsilon$, we have
\begin{displaymath}
   \lim_{\delta\rightarrow 0}\frac{M}{\sqrt{-\log \delta}} = C
\end{displaymath}
for some finite constant $C$.

2) Theorem ~\ref{accurate} does not follow from Lemma ~\ref{lem41}
directly simply because for any constant $M>0$ the Lyapunov constant
of $W$ in the set $\Omega_{M\epsilon^2}(W)$ becomes O($\epsilon)$
instead of being O(1).

\end{rem}

Next we estimate the lower bound of concentration of $\mu_{\epsilon}$.

\begin{lem}\label{A1} There is a constant $r_0 > 0$ such that $\sum_{i,j = 1}^{n}a_{ij}(x)W_{ij}(x)$ is
uniformly positive in $ B( \mathcal{A}, r_0)$.
\end{lem}

\begin{proof}
First we note that the Hessian matrix $\mathbf{H}(x) := (
W_{ij}(x))$ of $W(x)$ must be  positive semidefinite for all $x \in
\mathcal{A}$. For otherwise, there is $x_0\in \cal A$ such that
$\mathbf{H}(x_0)$ has negative eigenvalue.  It then follows from the
$C^{2}$ smoothness of $W(x)$ that $W(x)$ must take a negative value
at some $x\in \cal N$ where $x-x_{0}$ is an eigenvector
corresponding to the negative eigenvalue of
$\mathbf{H}(x_0)$. This is a contradiction because $W(x)$ must be
everywhere non-negative in $\cal N$.

Since $A(x)=(a_{ij}(x))$ is everywhere positive definite in
$\bar{\cal N}$, all its eigenvalues in $\cal N$ are bounded below by
a positive constant $\lambda_0$. For any $x_0\in \cal A$, since
$W(x) \geq L_{1} \mathrm{dist}(x, \mathcal{A})^{2}$, $x\in \cal N$,
Taylor expansion of $W$ at $x_0$ shows that at least one eigenvalue
of $\mathbf{H}(x_0)$ must be  positive. Consequently,
$$
  \sum_{i,j = 1}^{n}a_{ij}(x_{0}) W_{ij}(x_0) = \mathrm{trace}
  (A(x_0)\mathbf{H}(x_0)) \ge \lambda_0 \mathrm{trace}
  (\mathbf{H}(x_0))>0.
$$
The proposition simply follows from the continuity of $\sum_{i,j =
1}^{n}a_{ij}(x) W_{ij}(x)$.

\end{proof}

\begin{thm}
\label{levelalpha} If both {\bf  H$^0$)} and {\bf H$^1$)} hold, then
$$
  \lim_{\epsilon\rightarrow 0} \mu_{\epsilon}(\{x : \,
  \epsilon^{1+\alpha}\leq \mathrm{dist}(x, \mathcal{A}) \leq
  \epsilon^{1-\alpha}  \}) = 1
$$
for any $0<\alpha <1$.
\end{thm}

\begin{proof} By Lemma ~\ref{conditions}, there is a constant
  $\epsilon_{0}>0$ such that both $\{x \, | \mathrm{dist}(x,
  \mathcal{A}) \leq \epsilon^{1-\alpha} \} \subseteq
  \Omega_{\epsilon^{2-\alpha}(W)}$ and $\{x \, | \mathrm{dist}(x,
  \mathcal{A}) \leq \epsilon^{1+\alpha} \} \supseteq
  \Omega_{\epsilon^{2+3\alpha}(W)}$ hold for all $\epsilon \in (0, \epsilon_{0})$. Thus it suffices to prove that
\begin{eqnarray}
&&\lim_{\epsilon\rightarrow 0}\mu_{\epsilon}(\Omega_{\epsilon^{2+\delta}}(W)) =
0,\label{c12}\\
 &&\lim_{\epsilon\rightarrow
   0}\mu_{\epsilon}(\Omega_{\epsilon^{2-\delta}}(W)) = 1.\label{c13}
\end{eqnarray}
for any fixed $\delta > 0$.

\medskip

Equation \eqref{c13} follows from Theorem ~\ref{accurate}
immediately. We will prove equation \eqref{c12}.

\medskip

Fix a $\rho_{0} > 0$ such that $\Omega_{\rho_{0}}(W) \subseteq
\mathcal{N}$ and $\Gamma_{\rho_{0}} \cap \mathcal{A} =
\emptyset$. Consider $f(\rho) = \mu_{\epsilon}( \Omega_{\rho}(W))$ for
$\rho \in [0, \rho_{0}]$. Assume, for the sake of contradiction, that there
is a constant $\sigma > 0$ such that $f(\epsilon^{2+\delta}) \geq
\sigma > 0$ for any sufficient small $\epsilon > 0$.

\medskip

\medskip
We have by Theorem ~\ref{deri} that
\begin{equation}
\label{3-21}
  \int_{\Gamma_{\rho}(W)} \frac{u(x)}{|\nabla W(x)|} \mathrm{d}s = f'(\rho)
  \,.
\end{equation}

\medskip

Let  $\bar{\lambda} = \sup_{x \in \Omega_{\rho^{0}}(W)}
\lambda_{M}(x)$ with $\lambda(x)$ being the largest eigenvalue of
matrix $A(x)$ for each $x \in \Omega_{\rho^{0}}(W)$. Let $u =
u_{\epsilon}$ be the density function of $\mu_{\epsilon}$. It follows from Lemma \ref{conditions} and \eqref{3-21} that for each $0<\rho<\rho_{0}$, inequality
\begin{eqnarray*}
&&\int_{\Gamma_{\rho}(W)} \frac{1}{2}\epsilon^{2}\left (\sum_{i,j = 1}^{n}
a_{ij}(x)\partial_{i} W(x) \nu_{j}(x)\right
)u(x) \mathrm{d}s \\
&=&  \int_{\Gamma_{\rho}(W)}
\frac{1}{2}\epsilon^{2}\left (\sum_{i,j =
    1}^{n}a_{ij}(x)\partial_{i}
  W(x) \partial_{j} W(x)\right )\frac{u(x)}{|\nabla W(x)|} \mathrm{d}s\\
&\leq&\int_{\Gamma_{\rho}(W)} \frac{1}{2}\epsilon^{2} \bar{\lambda}
|\nabla W(x)|^{2} \frac{u(x)}{|\nabla W(x)|} \mathrm{d}s \\
&\leq& \int_{\Gamma_{\rho}(W)} \frac{1}{2}\epsilon^{2} \bar{\lambda} \frac{K_{2}^{2}}{L_{1}}W(x) \frac{u(x)}{|\nabla W(x)|} \mathrm{d}s\\
& \leq & \epsilon^{2}C_{1} \rho f'(\rho)
\end{eqnarray*}
holds for some positive constant $C_{1} < \infty$.

\medskip

By Lemma ~\ref{conditions}, we have the inequality
\begin{equation}
\label{eq3-20}
  \sum_{i =    1}^{n}f_{i}(x)\partial_{i } W(x) \geq -\kappa |\nabla W(x)|^{2}
  \geq -\kappa \frac{K_{2}^{2}}{L_{1}}W(x) \,.
\end{equation}
It then follows from Lemma ~\ref{A1} and \eqref{eq3-20} that there are positive
constants $p, C_{2}$ and $\epsilon_{0}$ such that
$$
  \frac{1}{2}\epsilon^{2}\sum_{i,j = 1}^{n}a_{ij}(x)W_{ij}(x) + \sum_{i =
    1}^{n}f_{i}(x)\partial_{i } W(x) \geq p\epsilon^{2}
$$
for all $x \in \Omega_{C_{2}\epsilon^{2}}(W)$ and $\epsilon
\in (0, \epsilon_{0})$. Without loss of generality, we make
$\epsilon$  sufficiently small such that $\rho_{0} >
C_{2}\epsilon^{2}$. Then for each $\rho \leq C_{2}\epsilon^{2}$
there holds
 $$
   \int_{\Omega_{\rho}(W)} \left (\frac{1}{2}\epsilon^{2}\sum_{i,j = 1}^{n}a_{ij}(x)W_{ij}(x) + \sum_{i =
     1}^{n}f_{i}(x)\partial_{i } W(x)\right )u \mathrm{d}x \geq
   p f(\rho)\epsilon^{2} \,.
 $$

\medskip

Since by Theorem 2.2,
\begin{eqnarray*}
  &&\int_{\Omega_{\rho}(W)} \left (\frac{1}{2}\epsilon^{2}\sum_{i,j = 1}^{n}a_{ij}(x)W_{ij}(x) + \sum_{i =
     1}^{n}f_{i}(x)\partial_{i } W(x)\right
 )u \mathrm{d}x \\
&=& \int_{\Gamma_{\rho}(W)} \frac{1}{2}\epsilon^{2}\left
   (\sum_{i,j = 1}^{n}
a_{ij}(x)\partial_{i} W(x)
\nu_{j}(x)\right )u(x) \mathrm{d}s \,.
\end{eqnarray*}
we conclude that
$$
  pf(\rho)\epsilon^{2} \leq \epsilon^{2}C_{1}\rho f'(\rho)
$$
for each $0 < \rho \leq C_{2}\epsilon^{2}$. Thus
\begin{equation}
\label{3-22}
   \frac{f'(\rho)}{f(\rho)} \geq \frac{p}{C_{1}\rho} \,.
\end{equation}

Integrating \eqref{3-22} from $\epsilon^{2+\delta}$ to
$C_{2}\epsilon^{2}$ yields
 $$
   \log f(C_{2}\epsilon^{2}) - \log \sigma \geq
   \frac{p}{C_{1}}( \log C_{2} - \delta\log \epsilon ) \,.
 $$

\medskip

 As $\epsilon \rightarrow 0$, we have $f(C_{2}\epsilon^{2})
 \rightarrow \infty$. This contradicts to the fact that $f(\rho) \leq 1$. Hence $f(\epsilon^{2+\delta}) = \mu_{\epsilon}(\Omega_{\epsilon^{2+\delta}}) \rightarrow 0$. This completes the proof.

\end{proof}

\begin{rem}
Theorem ~\ref{levelalpha} says that the density function of
$\mu_{\epsilon}$ cannot be ``too narrow''  because
almost all the mass of $\mu_{\epsilon}$ is located in the set $\{x :
\,
  \epsilon^{1+\alpha}\leq \mathrm{dist}(x, \mathcal{A}) \leq
  \epsilon^{1-\alpha}  \}$.
\end{rem}

\subsection{Mean square displacement}    The concentration of $\mu_{\epsilon}
$ can be more concretely measured by the {\it mean
square displacement}  defined by
\[
   \label{MSD}
   V(\epsilon) = \int_{\mathbb{R}^{n}}
   \mathrm{dist}^{2}(x,\mathcal{A}) \mathrm{d}\mu_{\epsilon}.
\]





\medskip

The following theorem gives bounds
of the mean square displacement.

\begin{thm}\label{MSDbounds}  If both {\bf H$^0$)} and
{\bf H$^1$)} hold, then there are constants $V_{1}, V_{2},\epsilon_0
> 0$  independent of $\epsilon$ such that
$$
V_{2}\epsilon^{2}\le V(\epsilon) \leq V_{1}\epsilon^{2},\quad
\epsilon\in (0, \epsilon_{0}).
$$

\end{thm}
\begin{proof}
Fix a $\rho_0>0$ such that $\Omega_{\rho_0}(W) \subset \mathcal{N}$
and $\Gamma_{\rho_0}(W)\cap\cal A = \emptyset$. We have by condition
{\bf H$^1$)} and Lemma~\ref{lem31} that there is an $\epsilon_0\in (0,\epsilon^*)$ sufficiently
small such that
\begin{equation}\label{c7}
\mu_{\epsilon}(\mathbb{R}^{n}\setminus \Omega_{\rho_0}(W)) <
\epsilon^{2},\quad \epsilon\in (0, \epsilon_{0}).
\end{equation}
Consider the function
\begin{displaymath}
   G(\rho) = \int_{\Omega_{\rho_0}\backslash \Omega_{\rho}}|\nabla W(x)|^{2}
    \mathrm{d}\mu_\epsilon,\quad \rho\in [0,\rho_0].
\end{displaymath}
Then it follows from Lemma~\ref{conditions} that
\begin{align*}
& \frac{1}{K_{2}^{2}}G(0)\le
\int_{\Omega_{\rho_0}(W)}\mathrm{dist}^{2}(x,\mathcal{A})\mathrm{d}\mu_\epsilon
\le V(\epsilon)\\
&=
  \int_{\Omega_{\rho_0}(W)}\mathrm{dist}^{2}(x,\mathcal{A})\mathrm{d}\mu_\epsilon +
  \int_{\mathbb{R}^{n}\backslash
    \Omega_{\rho_0}(W)}\mathrm{dist}^{2}(x,\mathcal{A})\mathrm{d}\mu_\epsilon
  \\
 &\leq
  \frac{1}{K_{1}^{2}}G(0) + \int_{\mathbb{R}^{n}\backslash
    \Omega_{\rho_0}(W)}\mathrm{dist}^{2}(x,\mathcal{A})\mathrm{d}\mu_\epsilon,
\end{align*}
where $K_1,K_2$ are as in Lemma~\ref{conditions}.

We first estimate an upper bound of   $G(0)$ in term of
$\epsilon^{2}$.  Let $F(\rho) =
\mu_{\epsilon}(\Omega_{\rho_0}(W)\backslash \Omega_{\rho}(W))$. Then
it follows from equation \eqref{c4} that
\begin{equation}\label{c55}
  F(\rho) \leq F(\rho_{1})
  e^{-\frac{2\gamma_{0}C_{1}}{\epsilon^{2}\bar{\lambda}}(\rho-\rho_{1})}
  \le e^{-\frac{2\gamma_{0}C_{1}}{\epsilon^{2}\bar{\lambda}}(\rho-\rho_{1})},
  \quad \rho\in (\rho_{1},\rho_0),
\end{equation}
for all $0<\epsilon\ll 1$, where $C_{1},\bar{\lambda}, \gamma_{0}$
are constants independent of $\epsilon$ and $\rho_{1} =
C_{2}\epsilon^{2}$ for some constant $C_{2}$ independent of
$\epsilon$. Since by Lemma~\ref{conditions}, $|\nabla W|^{2} \leq
K_{2}^{2}\rho/L_{1}$, we have by \eqref{c55} and Theorem ~\ref{deri}
that
\begin{eqnarray*}
 &&  \int_{\Omega_{\rho_0}(W)\backslash \Omega_{\rho_{1}}(W)} |\nabla
   W(x)|^{2}\mathrm{d}\mu_\epsilon \leq -\int_{\rho_{1}}^{\rho_0}\frac{K_{2}^{2}}{L_{1}}\rho
   F'(\rho)\mathrm{d}\rho\\
& =& \left . -\frac{K_{2}^{2}}{L_{1}}\rho F(\rho) \right
   |^{\rho_0}_{\rho_{1}} +
   \frac{K_{2}^{2}}{L_{1}}\int_{\rho_{1}}^{\rho_0}F(\rho)\mathrm{d}\rho\\
&\leq& \frac{K_{2}^{2}}{L_{1}}\rho_{1}F(\rho_{1}) +
   \frac{K_{2}^{2}}{L_{1}}\int_{\rho_{1}}^{\infty}e^{-\frac{2\gamma_{0}
    C_{1}}{\epsilon^{2}\bar{\lambda}}s}\mathrm{d}s
\leq \frac{K_{2}^{2}}{L_{1}}(C_{2}\epsilon^{2} +
\frac{\bar{\lambda}\epsilon^{2}}{2\gamma_{0}C_{1}}e^{- \frac{2
  \gamma_{0}C_{1}\rho_{1}}{\epsilon^{2}\bar{\lambda}}}) :=  E_2\epsilon^2
\end{eqnarray*}
for all $0<\epsilon\ll 1$. By a simple  calculation, we also have
\begin{displaymath}
   \int_{\Omega_{\rho_{1}}(W)}|\nabla W(x)|^{2}\mathrm{d}\mu_\epsilon \leq M_{2}\epsilon^{2}
\end{displaymath}
as $0<\epsilon\ll 1$, where $M_2>0$ is a constant independent of
$\epsilon$. Thus,
\begin{eqnarray}
\label{eq3.2-1}
   G(0) = \int_{\Omega_{\rho_0}(W)}|\nabla W(x)|^{2}\mathrm{d}\mu_\epsilon &=& \int_{\Omega_{\rho_0(W)\setminus
   \Omega_{\rho_1}(W)}}|\nabla W(x)|^{2}\mathrm{d}\mu_\epsilon
   + \int_{\Omega_{\rho_1}(W)}|\nabla W(x)|^{2}\mathrm{d}\mu_\epsilon \nonumber\\
   & \leq&
  (E_{2}+M_2)\epsilon^{2}, \,\quad 0<\epsilon\ll 1.
\end{eqnarray}



Next, we estimate an upper bound of
$\int_{\mathbb{R}^{n}\backslash
     \Omega_{\rho_0}(W)}\mathrm{dist}^{2}(x,\mathcal{A})\mathrm{d}\mu_\epsilon$. Let
     $R_{0}>0$ be as in {\bf H$^1$)}, i.e.,
\[
\mu_{\epsilon}(\R^n\setminus B(0,r))   \leq e^{-\frac{
r^{p}}{\epsilon^{2}}},\quad r\ge R_{0},
\]
for all $\epsilon \in (0,\epsilon_*)$. Without loss of generality,
we may assume that $R_0$ is sufficiently large such that
\[
\mathrm{dist}(x,\mathcal{A})<2|x|,\quad |x|\ge R_0.
\]
Then
\begin{eqnarray*}
\int_{\mathbb{R}^{n}\backslash
B(0,R_{0})}\mathrm{dist}^{2}(x,\mathcal{A})\mathrm{d}\mu_\epsilon
&\leq & \int_{\{|x| > R_{0}\}}4|x|^{2}u\mathrm{d}x \leq4 \sum_{k =
  k_{0}}^{\infty}e^{-\frac{k^{p}}{\epsilon^{2}}}(k+1)^{n+2}C(n)\\
&\leq& 4 e^{-\frac{k_{0}^{p}}{\epsilon^{2}}}\sum_{k =
  0}^{\infty}e^{-\frac{k^{p}}{\epsilon_{0}^{2}}}(k+k_{0}+1)^{n+2}C(n) \leq C_{5}e^{-\frac{k_{0}^{p}}{\epsilon^{2}}}
\end{eqnarray*}
for all $\epsilon\in (0, \epsilon_{*})$, where $C(n)$ is the volume of the
unit  sphere in $\mathbb{R}^{n}$, $k_{0}$ is the largest integer
smaller than $R_{0}$, and $C_{5}$ is a constant independent of
$\epsilon$. Thus, we can make $\epsilon_0$ sufficiently small such
that
\begin{equation}\label{c6}
\int_{\mathbb{R}^{n}\backslash
B(0,R_{0})}\mathrm{dist}^{2}(x,\mathcal{A})\mathrm{d}\mu_\epsilon<\epsilon^2,\quad
\epsilon\in (0, \epsilon_{0}).
\end{equation}
Using \eqref{c7}, we can make $\epsilon_0$ further small if
necessary such that
\begin{equation}
\label{eq3.2-2}
    \int_{B(0,R_{0})\backslash \Omega_{\rho_0}(W)}
   \mathrm{dist}^{2}(x,\mathcal{A})\mathrm{d}\mu_\epsilon <
   \epsilon^{2},\quad \epsilon\in (0, \epsilon_{0}).
\end{equation}
It now follows from \eqref{c6} and \eqref{eq3.2-2} that
\begin{eqnarray*}
   &&\int_{\mathbb{R}^{n}\backslash
     \Omega_{\rho_0}(W)}\mathrm{dist}^{2}(x,\mathcal{A})\mathrm{d}\mu_\epsilon\\
& =&
   \int_{\mathbb{R}^{n}\backslash B(0,R_{0})}
   \mathrm{dist}^{2}(x,\mathcal{A})\mathrm{d}\mu_\epsilon +
   \int_{B(0,R_{0})\backslash
     \Omega_{\rho_0}(W)}\mathrm{dist}^{2}(x,\mathcal{A})\mathrm{d}\mu_\epsilon<2\epsilon^2
\end{eqnarray*}
for all $\epsilon\in (0, \epsilon_{0})$. This, when combining with
 \eqref{eq3.2-1}, yields
that
\begin{displaymath}
   V(\epsilon) \leq (2+\frac{E_{2}}{K_{1}^{2}})\epsilon^{2} :=
   V_{1}\epsilon^{2},\quad 0<\epsilon\ll 1.
\end{displaymath}

Finally, we estimate a lower bound of $V(\epsilon)$. For each
$\epsilon\in (0,\epsilon^*)$, let $u=u_\epsilon$ be a density
function of $\mu_\epsilon$. Then by Theorem ~\ref{lem23}
\begin{eqnarray}\label{c8}
  &&\int_{\Omega_{\rho}(W)} \left (\frac{1}{2}\epsilon^{2}\sum_{i,j = 1}^{n}a_{ij}(x)W_{ij}(x) + \sum_{i =
    1}^{n}f_{i}(x)\partial_{i } W(x)\right
)u(x) \mathrm{d}x \\\nonumber
&=&
\int_{\Gamma_{\rho}(W)}\frac{1}{2}\epsilon^{2} \left (\sum_{i,j
  = 1}^{n}a_{ij}(x)\partial_{i} W(x)
\nu_{j}(x)\right )u(x) \mathrm{d}s,\quad \rho\in
  (0,\rho_0).
\end{eqnarray}
Let $r_0$ be as in Lemma~\ref{A1} and fix a $\rho_*\in (0,\rho_0)$
such that $\Omega_{\rho_{*}}(W)\subset B(\mathcal{A},r_0)$. Since
$F(\rho_0) = 0$, we have by \eqref{c55} and mean value theorem that
there is a $\rho^{*}\in (\rho_{*},\rho_0)$ such that $F'(\rho^{*})
\leq e^{-\frac{\beta}{\epsilon^{2}}}(\rho_0-\rho_{*})^{-1}$ as
$0<\epsilon<\epsilon_0$, where
$\beta=\frac{2\gamma_{0}C_{1}}{\bar{\lambda}}(\rho_*-\rho_{1})$.  By
{\bf H$^1$)} and Lemma~\ref{lem31}, we can make $\epsilon_0$ further
small if necessary such that  $\mu_{\epsilon}(\Omega_{\rho_{*}}(W))
\geq 1 - o(\epsilon^{2})$ as $\epsilon\in (0, \epsilon_{0})$. It
follows from Lemma~\ref{A1} that
\begin{equation}\label{c9}
 \int_{\Omega_{\rho^{*}}(W)} \frac{1}{2}\epsilon^{2}\sum_{i,j =
    1}^{n}a_{ij}(x)W_{ij}(x) u(x) \mathrm{d}x\ge  \int_{\Omega_{\rho_{*}}(W)} \frac{1}{2}\epsilon^{2}\sum_{i,j =
    1}^{n}a_{ij}(x)W_{ij}(x) u(x) \mathrm{d}x \geq C_{4}\epsilon^{2}
\end{equation}
as $0<\epsilon<\epsilon_0$, where  $C_{4} > 0$ is a constant
independent of $\epsilon$. By Lemma~\ref{conditions}, we also have
\begin{equation}\label{c10}
\int_{\Omega_{\rho^{*}}(W)} \sum_{i =
  1}^{n}f_{i}(x)\partial_{i} W(x) u
\mathrm{d}x \geq -\kappa \int_{\Omega_{\rho^*}(W)} |\nabla W(x)|^{2}u(x)
\mathrm{d}x \geq - \kappa G(0).
\end{equation}
Let $C_{5} = \frac 12\sup_{x\in \mathcal{N}} \sum_{i,j =
  1}^{n}a_{ij}\partial_{i} W \partial_{j} W $ and assume without loss of generality
  that
$C_{5}e^{-\frac{\beta}{\epsilon^{2}}}(\rho_0-\rho_{*})^{-1} < C_{4}/2$,
$\epsilon\in (0, \epsilon_{0})$. It follows from Theorem~\ref{deri} that
\begin{eqnarray}
&&  \int_{\Gamma_{\rho^*}(W)} \frac{1}{2}\epsilon^{2} \left (\sum_{i,j  = 1}^{n}a_{ij}(x)\partial_{i} W(x)
  \nu_{j}(x) \right )u(x) \mathrm{d}s
\\\nonumber
&=& \int_{\Gamma_{\rho^*}(W)} \frac{1}{2}\epsilon^{2} (\sum_{i,j  =
    1}^{n}a_{ij}(x)\partial_{i} W(x) \partial_{j} W(x)) \frac{u(x)}{|\nabla
    W(x)|}
  \mathrm{d} s \nonumber\\
  &\leq& \epsilon^{2}C_{5}\int_{\Gamma_{\rho^*}(W)}  \frac{u(x)}{|\nabla
    W(x)|}
  \mathrm{d} s= \epsilon^{2}C_{5}F'(\rho^*) \nonumber\\
&\leq& \epsilon^{2}C_{5}
e^{-\frac{\beta}{\epsilon^{2}}}(\rho_0-\rho_{*})^{-1}\le \frac{C_4}2
\epsilon^2,\label{c11}
\end{eqnarray}
 as
$\epsilon\in (0, \epsilon_{0})$. Now, \eqref{c8}-\eqref{c11} yield
$$
  G(0) \geq \frac{C_{4}}{2 \kappa} \epsilon^{2},
$$
which implies
$$
V(\epsilon)\ge \frac{C_4}{2\kappa K_2^2}\epsilon^2:= V_2\epsilon^2
$$
as $\epsilon\in (0, \epsilon_{0})$.  This completes the proof.
\end{proof}

\begin{rem}
The mean square displacement is a natural extension of the variance.
Consider $\mathcal{A}_{1} = \{0\}$ and  the Gaussian measure
$\nu_{\epsilon}$ with mean $0$ and variance $\epsilon^{2}$. Then it
is easy to see that Theorems~\ref{levelalpha}, ~\ref{MSDbounds} hold
for $\nu_{\epsilon}$ and $\mathcal{A}_{1}$. This is to say that, by
assuming {\bf H$^{0}$)} and {\bf H$^{1}$)}, the concentration of
$\mu_{\epsilon}$ is Gaussian-like.
\end{rem}

\section{Entropy-dimension relationship}

 In this section, we will  investigate the connection
between the differential entropy of  stationary measures of \eqref{FPE1} and the
dimension of $\mathcal{A}$. This connection will be used  in the
second part of the series.

Let $\mu$ be a probability measure on $\mathbb{R}^{n}$ with a
density function $\xi(x)$.  We recall that the relative entropy of
$\mu$ with respect to Lebusgue measure, or the {\it differential
entropy} of $\mu$ is defined as
\begin{equation}
   \label{Ent}
   \mathcal{H}(\mu) =
   -\int_{\mathbb{R}^{n}}\xi(x)\log \xi(x)
   \mathrm{d}x \,.
\end{equation}

\medskip

\subsection{Regularity of sets and measures} To
establish the connection between the entropy of a stationary measure
$\mu_\epsilon$ of \eqref{FPE1} and the dimension of $\mathcal{A}$,
we will require $\cal A$ be a regular set and $\mu_\epsilon$ be a
regular measure with respect to $\cal A$.

A set $A \subset \mathbb{R}^{n}$ is called a {\it regular set} if
\begin{displaymath}
   \limsup_{r \rightarrow 0} \frac{\log m(B(A,r))}{\log r}  = \liminf_{r\rightarrow 0}
   \frac{\log m(B(A,r))}{\log r} = n-d
\end{displaymath}
for some $d\geq 0$, where    $m(\cdot)$ denotes the Lebesgue measure
on $\mathbb{R}^{n}$. It is easy to check that $d$ is the
  Minkowski dimension of $A$.  Regular sets form a large class that includes
smooth manifolds and even fractal sets like Cantor sets. However,
not all measurable sets are regular (see \cite{pesin1997dimension}
for details).

Assume that \eqref{ODE1} admits a global attractor $\cal A$ and the
Fokker-Planck equation \eqref{FPE1} admits a  stationary measure
$\mu_\epsilon$  for each  $\epsilon\in (0,\epsilon_*)$. The family
$\{\mu_\epsilon\}$
 of
stationary measures is said to be {\it regular with respect to
$\mathcal{A}$}  if for any $\delta > 0$ there are constants $K$, $C$
and a family of approximate funtions $u_{K,\epsilon}$ supported on $B( \mathcal{A}, K\epsilon)$ such that for all $\epsilon \in (0, \epsilon^*)$,
\begin{itemize}
  \item[a)] \begin{equation}\label{regular}
   \inf_{B(\mathcal{A},K\epsilon) }(u_{K, \epsilon}(x))\geq C
   \sup_{B(\mathcal{A},K\epsilon)}(u_{K, \epsilon}(x))  \,;
\end{equation}
and
\item[b)]
$$
 \|u_{\epsilon}(x) - u_{K, \epsilon}(x) \|_{L^{1}}  \leq \delta \,.
$$
\end{itemize}

\medskip

The following propositions give some examples of regular stationary
measures.
\medskip

\begin{pro}
Assume equation \eqref{SDE1} has the form
\begin{equation}\label{gradient}
  \mathrm{d}X_{t} = -\nabla U(X) \mathrm{d}t + \epsilon
  \mathrm{d}W_{t},\quad X\in\R^n,
\end{equation}
where $U\in C^{2}(\R^n)$  is such that $U(x)\ge \beta
\log |x|$ as $|x|\gg 1$ for some positive constant $\beta$.  Then
the family of stationary measures of the Fokker-Planck equations
associated with \eqref{gradient} is regular with respect to $\cal A$
as $0<\epsilon\ll 1$.
\end{pro}
\begin{proof}  For each  $\epsilon\ll1 $,  the Fokker-Planck equation
associated with \eqref{gradient} admits a unique stationary measure
$\mu_\epsilon$ which actually  coincides with the Gibbs measure
with density
\begin{equation}\label{c14}
  \frac{1}{K}e^{-\frac{U(x)}{\epsilon^{2}}} \,,
\end{equation}
where $K$ is the normalizer (see e.g. \cite{huang5} and
references therein). The regularity of the family
$\{\mu_\epsilon\}$ thus follows easily from \eqref{c14} and the
definition.
\end{proof}
\medskip

\begin{pro} Consider \eqref{SDE1} and assume that {\bf H$^1$)} holds. If $\mathcal{A}$ is an equilibrium, then the
the family $\mu_{\epsilon}$ is regular with respect to $\cal A$.
\end{pro}
\begin{proof} Without loss of generality, we assume
that  $\mathcal{A} = \{0\}$.  It follows from {\bf H$^1$)}
and the WKB expansion (see \cite{ludwig1975persistence, day1985some, day1994regularity}) that there is
 a function $W\in C^2(\R^n)$,  called quasi-potential
function, such that the density function of
$\mu_\epsilon$ for each $\epsilon\in (0,\epsilon^*)$ has the form
\begin{equation}\label{c15}
 \frac{1}{K}z(x)e^{-\frac{W(x)}{\epsilon^{2}}} + o(\epsilon^{2}) \,,
\end{equation}
where $K$ is the normalizer and $z\in C(\R^n)$. It is easy to see
that $u_\epsilon(x)$ is regular.  The regularity of the family
$\{\mu_\epsilon\}$ then follows  from \eqref{c15} and the
definition.
\end{proof}

\medskip

In many biological applications, WKB expansion as in \eqref{c15} is
assumed \cite{newby2011asymptotic, bressloff2014stochastic, maier1993escape}. If the family of stationary measures satisfies \eqref{c15},
then it must be regular with respect to $\mathcal{A}$. However, if
$\mathcal{A}$ is not an equilibrium, verifying \eqref{c15} is
difficult in general. Still, although there are some
technical huddles, proving that a stationary measure is regular with
respect to the global attractor is possible in many cases.

If $\mathcal{A}$ is a limit cycle on which $f$ is everywhere
non-vanishing, then equation \eqref{SDE1} can be linearized in the
vicinity of $\mathcal{A}$. The solution of the linearized equation can
be explicitly given. Therefore the density function of
$\mu_{\epsilon}$ can be estimated via probabilistic approaches. We will prove in our future work that the
family $\mu_{\epsilon}$ is regular with respect to the limit cycle
$\mathcal{A}$.  In addition, we conjecture that when
   {\bf H$^{1}$)} holds for equation \eqref{SDE1} and equation
   \eqref{ODE1} admits an SRB measure, the family
 $\mu_{\epsilon}$ is regular with respect to $\mathcal{A}$ under
 suitable conditions.



\subsection{Entropy and dimension}
The main theorem of this subsection is the following
  entropy-dimension inequality.
\begin{thm}
\label{EntDimThm}
Assume that {\bf H$^{0}$)} and {\bf H$^{1}$)} hold. If $\mathcal{A}$ is a regular set, then
\begin{equation}
   \label{EntDim}
   \liminf_{\epsilon\rightarrow
     0}\frac{\mathcal{H}(\mu_{\epsilon})}{\log \epsilon} \geq n - d \,,
\end{equation}
where $d$ is the Minkowski dimension of $\mathcal{A}$.
\end{thm}

 To prove Theorem \ref{EntDimThm}, the following three lemmas that
estimate the integral of $ u_{\epsilon}(x)\log
u_{\epsilon}(x) $ are useful. The first lemma
  estimates the integral of $u_{\epsilon}(x) \log u_{\epsilon}(x)$
  outside large spheres.
\medskip

\begin{lem}
\label{entout}
Let $l > 0$ be a fixed constant independent of $\epsilon$. If {\bf H$^{1}$)} holds, then there exist positive constants $\epsilon_{0}$, $R_{0}$ such that
$$
  \int_{|x|>R_{0}} u_{\epsilon}(x)\log u_{\epsilon}(x) \mathrm{d}x\geq
  -\epsilon^{l} \,,
$$
for all $\epsilon \in (0, \epsilon_{0})$.
\end{lem}
\begin{proof}
It follows from {\bf H$^{1}$)} that $\mu_{\epsilon}$ has the
tail bounds
$$
 \mu_{\epsilon}( \mathbb{R}^{n}\setminus B(0, |x|)) =  \int_{\mathbb{R}^{n}\setminus B(0, |x|)} u_{\epsilon}(x) \mathrm{d}x
  \leq e^{-\frac{|x|^{p}}{\epsilon^{2}}}
$$
for all $|x| >R_{0}$, where $R_{0}$ and $p$ are positive constants independent of $\epsilon$.

\medskip

 For each positive integer $k$, we denote
$\Omega_{k}=\{x : \, k \leq |x| < k+1 \}$. Let $k_{0}$ be the
smallest integer that is larger than $R_{*}$. Then for each $k >
k_{0}$, we have
$$
  \int_{\Omega_{k}}u_{\epsilon}(x) \mathrm{d}x \leq e^{-\frac{k^{p}}{\epsilon^{2}}} \,.
$$

\medskip

Let $\Omega_{k} = A_{k} \cup B_{k}$ where
$A_{k} = \{ x\in \Omega_{k} : \, u_{\epsilon}(x) \geq e^{-\frac{
  k^{p}}{\epsilon^{2}}} \}$, $B_{k} = \Omega_{k}\setminus A_{k}$. Since $x\log x$ decreases on the interval $(0 ,e^{-1})$, for sufficient small $\epsilon$ we have
$$
  \int_{A_{k}}u_{\epsilon}(x)\log u_{\epsilon}(x) \mathrm{d}x \geq -\frac{k^{p}}{\epsilon^{2}}\int_{A_{k}}u_{\epsilon}(x) \mathrm{d}x \geq
  -\frac{k^{p}}{\epsilon^{2}}e^{- \frac{k^{p}}{\epsilon^{2}}} =: a_{k}
$$
and
$$
  \int_{B_{k}} u_{\epsilon}(x)\log u_{\epsilon}(x) \geq \int_{B_{k}} -\frac{ k^{p}}{\epsilon^{2}}e^{-k^{p}/\epsilon^{2} } \mathrm{d}x \geq
  -\frac{\pi^{\frac{n}{2}}}{\Gamma(\frac{n}{2}+1)}k^{n}\frac{ k^{p}}{\epsilon^{2}}e^{-\frac{k^{p}}{\epsilon^{2}} } =:
  b_{k} \,,
$$
where $\Gamma(x)$ is the Gamma function.

\medskip

It is easy to see that for any $l > 0$ there is an $\epsilon_{0} > 0$ such that
$$
  \sum_{k=k_{0}}^{\infty} a_{k}+b_{k} \geq -\epsilon^{l}
$$
for all $\epsilon \in (0, \epsilon_{0})$. The proof is completed by letting $R_{0} = k_{0}$.
\end{proof}

\medskip
The second lemma bounds the integration of
  $u_{\epsilon}(x) \log u_{\epsilon}(x)$ over compact sets.

\begin{lem}
\label{entinside} Let $v(x)$ be a probability density function on
$\mathbb{R}^{n}$ and $\Omega \subseteq \mathbb{R}^{n}$ be a Lebesgue
measurable compact set. Then there is a $\delta_{0} > 0$ such that
for each $\delta \in (0, \delta_{0})$, if
$$
  \int_{\Omega}v(x) \mathrm{d}x \leq \delta
$$
then
$$
  \int_{\Omega} v(x)\log v(x) \mathrm{d}x \geq
  -2 \sqrt{\delta} \,.
$$
\end{lem}
\begin{proof}
The proof only contains elementary calculations. Let $\eta =
\sqrt{\delta}$ and write $\Omega$ into $\Omega = A \cup B$, where $A
= \{x \in \Omega  : \, v(x) > e^{-1/\eta} \}$ and $B = \Omega
\setminus A$. Then for every $\eta < 1$,  we have
$$
  \int_{A} v(x) \log v(x) \mathrm{d}x \geq -\frac{1}{\eta} \int_{A}v(x) \mathrm{d}x \geq -
  \frac{\delta}{\eta}
$$
and
$$
  \int_{B} v(x)\log v(x) \mathrm{d}x \geq -\frac{1}{\eta}e^{-\frac{1}{\eta}} \int_{B} \mathrm{d}x \geq -\frac{V}{\eta}e^{-\frac{1}{\eta}}
  \,,
$$
 where $V$ denotes the Lebesgue measure of $\Omega$.
\medskip

Let $\delta_{0} = ( \log V)^{-2}$. Then for any $\delta < \delta_{0}$,
we have
$$
    \int_{\Omega} v(x)\log v(x) \mathrm{d}x \geq
  -2 \sqrt{\delta} \,.
$$
\end{proof}

The upper bound of $u_{\epsilon}(x)$ is estimated in the
  following lemma.

\begin{lem}
\label{entbound}
If {\bf H$^{1}$)} holds, then there is a constant $\epsilon_{0} > 0$ such that $u_{\epsilon}(x) \leq \epsilon^{-2n+1}$ whenever $x \in \mathbb{R}^{n}$ and $\epsilon \in (0, \epsilon_{0})$.
\end{lem}

\begin{proof}

We first show that there are positive constants $p$, $\epsilon_{0}$
and $R_{*} < \infty$ independent of $\epsilon$ such that
\begin{equation}\label{(*)} u_{\epsilon}(x) \leq
e^{-\frac{|x|^{p}}{2\epsilon^{2}}}
\end{equation}
 for every $|x| > R_{*}$ and $\epsilon \in (0, \epsilon_{0})$.

\medskip

It follows from {\bf H$^{1}$)} that there exist positive constants $p$, $\epsilon_{*}$ and $R_{0}$ such that
$$
  \mu_{\epsilon}(B(0, r)) \leq e^{-\frac{r^{p}}{\epsilon^{2}}}
$$
for all $r > R_{0}$ and $\epsilon \in (0, \epsilon_{*})$. For the sake
of contradiction, we assume
$u_{\epsilon}(x_{0}) > e^{-|x_{0}|^{p}/2\epsilon^{2}}$ for some $x_{0} \in \mathbb{R}^{n}$ with $|x_{0}| > R_{0}+ 1$.
It follows from Lemma ~\ref{harnack} that there is a constant $C >0$ such that for any ball $B(x_{0},r)$ where $r < 1/4$, we have
$$
  \sup_{B(x_{0},r)} u_{\epsilon} \leq C \inf_{B(x_{0},r)} u_{\epsilon} \,,
$$
where $C = C_{0}(n)^{C_{1} + \nu r \epsilon^{-2}}$, $C_{0}$, $C_{1}$
and $\nu$ are constants independent of $\epsilon$. Let $r = \epsilon^{2}$ and $C_{*} = C_{0}(n)^{C_{1} + \nu}$. Then
\begin{equation}
\label{eq6-3}
  \int_{B(x_{0},r)}u_{\epsilon}(x) dx >
  \frac{\pi^{\frac{n}{2}}}{\Gamma(\frac{n}{2} + 1)}\frac{1}{C_{*}} \epsilon^{2n}  e^{-\frac{|x_{0}|^{p}}{2\epsilon^{2}}}\,.
\end{equation}

As $\epsilon$ approaches to $0$, $e^{\frac{|x_{0}|^{p}}{2\epsilon^{2}}}$
grows faster than any power of $\epsilon^{-1}$. Hence one can make $\epsilon$
sufficient small such that
$$
  \int_{B(x_{0},r)}u_{\epsilon}(x) dx > e^{-\frac{|x_{0}|^{p}}{\epsilon^{2}}} \,.
$$
This contradicts with {\bf H$^{1}$)}. Hence the claim holds for $p$,
$R_{0} = R_{*} + 1$ and some sufficiently small $\epsilon_{0}$.

\medskip

Next we consider the upper bound of $u_{\epsilon}(x)$ within
$B(0, R_{*})$.

\medskip

Assume, for the sake of contradiction, that $u_{\epsilon}(x_{1}) \geq
\epsilon^{-(2n+1)}$ at some point $x_{1} \in B(0, R_{*})$. Apply Lemma
~\ref{harnack} to $B(0,
R_{*} + 1)$. We have that for any $x \in B(0, R_{*})$ and $r \in (0, 1/4)$,
$$
  \sup_{B(x, r)} u_{\epsilon} \leq \hat{C} \inf_{B(x ,r)} u_{\epsilon} \,,
$$
where $\hat{C} = \hat{C}_{0}(n)^{\hat{C}_{1} + \hat{\nu} r
  \epsilon^{-2}}$, $\hat{C}_{0}$, $\hat{C}_{1}$ and $\hat{\nu}$ are constants independent of $\epsilon$ and $x$.

Let $r = \epsilon^{2}$ and consider the neighborhood $B(x_{1},r)$. By
Lemma ~\ref{harnack}, we have
$$
  \min_{B(x_{1},r)} u(x) \geq C_{3}\epsilon^{-(2n+1)} \,,
$$
where constant $C_{3}$ is independent of $\epsilon$. Thus, if
$\epsilon^{-1} >
C_{3}\frac{\pi^{\frac{n}{2}}}{\Gamma(1+\frac{n}{2})}$, then
$$
  \mu_{\epsilon}(B(x_{1},r)) >1 \,.
$$
This contradicts with the fact that $\mu_{\epsilon}$ is a probability
measure. Therefore,
\begin{equation}\label{(**)}
 u_{\epsilon}(x) \leq \epsilon^{-(2n+1)}, \quad x \in
  B(0, R_{*}) \,.
\end{equation}

\medskip

It now follows from \eqref{(*)} and \eqref{(**)} that
$u_{\epsilon}(x)$ is globally bounded by $\epsilon^{-2n-1}$ for
sufficient small $\epsilon > 0$. This completes the proof.

\end{proof}

Now we are ready to prove Theorem \ref{EntDimThm}.

\begin{proof}[Proof of Theorem \ref{EntDimThm}]

Let $\sigma \in (0, 1)$ be a small positive constant. Theorem
~\ref{accurate} implies that there are constants $M, \epsilon_{0} > 0$ such that
\begin{displaymath}
   \mu_{\epsilon}(B( \mathcal{A},M\epsilon)) \geq 1-\sigma
\end{displaymath}
for all $\epsilon \in (0, \epsilon_{0})$.

\medskip

Let
$$
F(u_{\epsilon}) = \int_{B(
\mathcal{A},M\epsilon)} u_{\epsilon}(x)\log u_{\epsilon}(x) \mathrm{d}x \,.
$$
Applying Lagrange multiplier with constraint $\int_{B(
\mathcal{A},M\epsilon)} u_{\epsilon} \mathrm{d}x = \mu_{\epsilon}(B(
\mathcal{A},M\epsilon))$, it is easy to see that
$F(u_{\epsilon})$ attains its minimum when $u_{\epsilon}$ is a
constant function on $B( \mathcal{A},M\epsilon)$. Thus
\begin{eqnarray*}
F( u_{\epsilon})&\geq&
\int_{B(\mathcal{A},C_{2}\epsilon)}u_{a}(x)\log u_{a}(x) \mathrm{d}x  \\
&\geq &
(1-\sigma)\log\frac{1-\sigma}{m(B(\mathcal{A},C_{2}\epsilon))} \\
&=& (1-\sigma)\log(1-\sigma) - (1-\sigma) \log m(B( \mathcal{A},M\epsilon))\,,
\end{eqnarray*}
where $u_{a}(x)$ is the constant function on $B( \mathcal{A},M\epsilon)$ with value $\frac{\mu_{\epsilon}(B( \mathcal{A},M\epsilon))}{ m(B(\mathcal{A},M\epsilon))}$.

\medskip

The regularity of  $\mathcal{A}$ implies that
\begin{displaymath}
   \lim_{r\rightarrow 0} \frac{\log m(B(\mathcal{A},r))}{\log
   r} = n-d \,.
\end{displaymath}

Thus
$$
  \lim_{\epsilon\rightarrow 0} \frac{\log m(B( \mathcal{A},M\epsilon))}{\log \epsilon} = \lim_{\epsilon\rightarrow 0} \frac{\log m(B( \mathcal{A},M\epsilon))}{\log M\epsilon} = n-d \,.
$$

\medskip

Next we estimate
$$
  \int_{\mathbb{R}^{n}\setminus B( \mathcal{A},M\epsilon)} u_{\epsilon}(x)\log u_{\epsilon}(x) \mathrm{d}x \,.
$$
It follows from Lemmas ~\ref{entout} and
~\ref{entinside} that there are positive constants $R_{*}$ and $\epsilon_{0}$, such that for all $\epsilon \in (0, \epsilon_{0})$, the integral of $u_{\epsilon} \log u_{\epsilon}$ on $B(0,R_{*})\setminus B( \mathcal{A}, C_{2}\epsilon)$ and $\mathbb{R}^{n}\setminus B(0,R_{*})$ are bounded from
below by $-2\sqrt{\sigma}$ and $-\epsilon^{2}$ respectively. Thus
$$
  \int_{\mathbb{R}^{n}\setminus B( \mathcal{A},M\epsilon)} u_{\epsilon}(x)\log u_{\epsilon} (x)\mathrm{d}x \geq -\epsilon^{2} - 2 \sqrt{\sigma}
$$
for $\epsilon \in (0, \epsilon_{0})$.
\medskip

 Now, for any $0<\sigma\ll 1$, some calculations show
that
\begin{eqnarray*}
&&\liminf_{\epsilon\rightarrow 0}\frac{\int_{\mathbb{R}^{n}}
  u_{\epsilon} (x)\log u_{\epsilon}(x) \mathrm{d}x}{-\log \epsilon} \\
&\geq
&\lim_{\epsilon\rightarrow 0}\frac{(1-\sigma)(\log(1-\sigma)-\log(m(B(\mathcal{A},M\epsilon))))}{-\log
\epsilon}  - \frac{\epsilon^{2}+2\sqrt{\sigma}}{-\log \epsilon}\\
&=&(1-\sigma)\lim_{\epsilon\rightarrow 0}\frac{\log
  m(B(\mathcal{A},M\epsilon))}{\log \epsilon}\\
&=&(1-\sigma)(n-d) \,.
\end{eqnarray*}
 Thus
\begin{displaymath}
   \liminf_{\epsilon\rightarrow 0}\frac{\mathcal{H}(u_{\epsilon})}{\log
     \epsilon}\geq n-d \,.
\end{displaymath}
This completes the proof.
\end{proof}
\medskip

In general, the reversed inequality of \eqref{EntDim}
\begin{equation}
  \label{entdimR}
   \limsup_{\epsilon\rightarrow 0}\frac{\mathcal{H}(\mu_{\epsilon})}{\log
     \epsilon}\leq n-d
\end{equation}
for $\mu_{\epsilon}$ cannot be shown by level set method. It can be
shown by Theorem ~\ref{levelalpha} and some calculation
that for sufficient small $\epsilon > 0$ the integral of
$u_{\epsilon}(x)$ on each level set $\Gamma_{\rho}(W)$ is bounded by
$\epsilon^{-1}$. However, the distributions of $u_{\epsilon}(x)$ on
each of the level sets are not clear.

\medskip

The theorem below gives some cases when \eqref{entdimR}
actually holds.

\medskip


\medskip

\begin{thm}
Assume {\bf H$^{0}$)} and {\bf H$^{1}$)} holds and the stationary measures $\mu_{\epsilon}$ are regular with respect to $\mathcal{A}$. Then
\begin{equation}
   \label{EndDimId}
   \lim_{\epsilon\rightarrow
     0}\frac{\mathcal{H}(\mu_{\epsilon})}{\log \epsilon} = n-d \,.
\end{equation}
\end{thm}
\begin{proof}
It follows from the definition of regular stationary measures with
respect to $\mathcal{A}$ that for any $\delta > 0$ there are
positive constants $K$ and $\epsilon_{0}$ and a family of
approximate functions $u_{K, \epsilon}$ such that $u_{K, \epsilon}$
approximates $u_{\epsilon}$ in the vicinity of $\mathcal{A}$. By the
definition of regular stationary measures, there is a positive
constant $C$ independent of $\epsilon$ such that
\begin{displaymath}
   \min(u_{K, \epsilon}(x))\geq C \max(u_{K, \epsilon}(x)) ;\quad
   x \in B(\mathcal{A},K\epsilon)  \,.
\end{displaymath}
This means that
$$
u_{K,\epsilon}(x) \leq \frac{C^{-1}}{m(B(\mathcal{A},K\epsilon) )}
$$
for all $x \in B(\mathcal{A},K\epsilon)$.

\medskip

Let $u_{1} = u_{\epsilon} - u_{K, \epsilon}$. Then by the convexity of
$x\log x$, we have
\begin{eqnarray*}
  &&\int_{\mathbb{R}^{n}} u_{\epsilon}(x) \log u_{\epsilon}(x) \mathrm{d}x
   =    \int_{\mathbb{R}^{n}} (u_{K, \epsilon}(x) + u_{1}(x))\log
  (u_{K,\epsilon}(x) + u_{1}(x)) \mathrm{d}x\\
  &\leq& \int_{\mathbb{R}^{n}}  (u_{K, \epsilon}(x) + |u_{1}(x)|)\log
  (u_{K,\epsilon}(x) + |u_{1}(x)|) \mathrm{d}x \\
&&+ 2 \int_{\mathbb{R}^{n}}
  |u_{1}(x)| |\log (u_{K, \epsilon}(x) + |u_{1}(x)|)|  \mathrm{d}x \\
  &=& 2\int_{\mathbb{R}^{n}}  \frac{u_{K, \epsilon}(x) + |u_{1}(x)|}{2}\log
  \frac{u_{K,\epsilon}(x) + |u_{1}(x)|}{2} \mathrm{d}x \\
&&+ 2 \int_{\mathbb{R}^{n}}
  |u_{1}(x)| |\log (u_{K, \epsilon}(x) + |u_{1}(x)|)|  \mathrm{d}x  + \log 2 \\
  &\leq& \int_{\mathbb{R}^{n}} u_{K, \epsilon}(x) \log u_{K, \epsilon}(x)
  \mathrm{d}x + \int_{\mathbb{R}^{n}} |u_{1}(x)| \log |u_{1}(x)|
  \mathrm{d}x \\
&&+ 2 \int_{\mathbb{R}^{n}}
  |u_{1}(x)| |\log (u_{K, \epsilon}(x) + |u_{1}(x)|)|  \mathrm{d}x  + \log 2\\
&:=& \int_{\mathbb{R}^{n}} u_{K, \epsilon}(x) \log u_{K, \epsilon}(x)
  \mathrm{d}x + I_{1}
\end{eqnarray*}

It follows from Lemma ~\ref{entbound} that $u_{\epsilon}(x)$ is bounded
from above by $\epsilon^{-2n-1}$ globally, so are $u_{K,\epsilon}$ and $u_{1}$. Hence we have
$$
  I_{1} \leq 3(2n+1)(-\log\epsilon)\delta + \log 2 \,.
$$

\medskip

Take the limit $\epsilon\rightarrow 0$. The regularity of
$\mathcal{A}$ and the upper bound of $u_{K, \epsilon}$ together yield that
\begin{displaymath}
      \limsup_{\epsilon\rightarrow
     0}\frac{\mathcal{H}(\mu_{\epsilon})}{\log \epsilon} = \limsup_{\epsilon\rightarrow
     0}\frac{\int_{\mathbb{R}^{n}} u_{\epsilon}(x)\log u_{\epsilon}(x) \mathrm{d}x}{-\log \epsilon} \leq
   n-d+3\delta (2n+1) \,.
\end{displaymath}
The above inequality holds for any $\delta>0$. Hence
\begin{displaymath}
         \limsup_{\epsilon\rightarrow
     0}\frac{\mathcal{H}(\mu_{\epsilon})}{\log \epsilon} \leq n-d \,.
\end{displaymath}

 Combining this with  Theorem ~\ref{EntDimThm}, the proof
is completed.

\end{proof}

\begin{rem}
The entropy-dimension inequality and entropy-dimension equality will
be used in the second part of the series  when we
discuss the properties of degeneracy and complexity.
\end{rem}

\bibliographystyle{amsplain}
\bibliography{myref}

\providecommand{\bysame}{\leavevmode\hbox to3em{\hrulefill}\thinspace}
\providecommand{\MR}{\relax\ifhmode\unskip\space\fi MR }
\providecommand{\MRhref}[2]{%
  \href{http://www.ams.org/mathscinet-getitem?mr=#1}{#2}
}
\providecommand{\href}[2]{#2}
\begin{thebibliography}{10}

\bibitem{ao2007existence}
P.~Ao, C.~Kwon, and H.~Qian, \emph{On the existence of potential landscape in
  the evolution of complex systems}, Complexity \textbf{12} (2007), no.~4,
  19--27.

\bibitem{athreya2012propagating}
A.~Athreya, T.~Kolba, and J.C. Mattingly, \emph{Propagating lyapunov functions
  to prove noise--induced stabilization}, Electron. J. Probab \textbf{17}
  (2012), no.~96, 1--38.

\bibitem{bogachev2001regularity}
V.I. Bogachev, N.V. Krylov, and M.~R{\"o}ckner, \emph{On regularity of
  transition probabilities and invariant measures of singular diffusions under
  minimal conditions}, Communications in Partial Differential Equations
  \textbf{26} (2001), no.~11-12, 2037--2080.

\bibitem{bogachev2009elliptic}
\bysame, \emph{Elliptic and parabolic equations for measures}, Russian
  Mathematical Surveys \textbf{64} (2009), 973.

\bibitem{bogachev2012positive}
V.I. Bogachev, M.~R{\"o}ckner, and S.V. Shaposhnikov, \emph{On positive and
  probability solutions to the stationary fokker-planck-kolmogorov equation},
  Doklady Mathematics, vol.~85, Springer, 2012, pp.~350--354.

\bibitem{bressloff2014stochastic}
P.C. Bressloff, \emph{Stochastic processes in cell biology}, vol.~41, Springer,
  2014.

\bibitem{clark2011degeneracy}
E.~Clark, A.~Nellis, S.~Hickinbotham, S.~Stepney, T.~Clarke, M.~Pay, and
  P.~Young, \emph{Degeneracy enriches artificial chemistry binding systems},
  European Conference on Artificial Life, 2011.

\bibitem{day1994regularity}
M.V. Day, \emph{Regularity of boundary quasi-potentials for planar systems},
  Applied mathematics \& optimization \textbf{30} (1994), no.~1, 79--101.

\bibitem{day1985some}
M.V. Day and T.A. Darden, \emph{Some regularity results on the ventcel-freidlin
  quasi-potential function}, Applied Mathematics and Optimization \textbf{13}
  (1985), no.~1, 259--282.

\bibitem{delignieres2013degeneracy}
D.~Deligni{\`e}res and V.~Marmelat, \emph{Degeneracy and long-range
  correlations}, Chaos: An Interdisciplinary Journal of Nonlinear Science
  \textbf{23} (2013), no.~4, 043109.

\bibitem{delignieres2011degeneracy}
D.~Deligni{\`e}res, V.~Marmelat, and K.~Torre, \emph{Degeneracy and long-range
  correlation: a simulation study}, BIO Web of Conferences, vol.~1, EDP
  Sciences, 2011, p.~00020.

\bibitem{dembo2009large}
A.~Dembo and O.~Zeitouni, \emph{Large deviations techniques and applications},
  vol.~38, Springer Verlag, 2009.

\bibitem{edelman2001degeneracy}
G.M. Edelman and J.A. Gally, \emph{Degeneracy and complexity in biological
  systems}, Proceedings of the National Academy of Sciences \textbf{98} (2001),
  no.~24, 13763.

\bibitem{edelman1978mindful}
G.M. Edelman and V.B. Mountcastle, \emph{The mindful brain: Cortical
  organization and the group-selective theory of higher brain function.},
  Massachusetts Inst of Technology Pr, 1978.

\bibitem{feinberg1987chemical}
M.~Feinberg, \emph{Chemical reaction network structure and the stability of
  complex isothermal reactors--i. the deficiency zero and deficiency one
  theorems}, Chemical Engineering Science \textbf{42} (1987), no.~10,
  2229--2268.

\bibitem{feinberg1995existence}
\bysame, \emph{The existence and uniqueness of steady states for a class of
  chemical reaction networks}, Archive for Rational Mechanics and Analysis
  \textbf{132} (1995), no.~4, 311--370.

\bibitem{FW}
M.I. Fre{\u\i}dlin and A.D. Wentzell, \emph{Random perturbations of dynamical
  systems}, vol. 260, Springer Verlag, 1998.

\bibitem{gilbarg2001elliptic}
D.~Gilbarg and N.S. Trudinger, \emph{Elliptic partial differential equations of
  second order}, Springer Verlag, 2001.

\bibitem{gopalkrishnan2013lyapunov}
M.~Gopalkrishnan, \emph{On the lyapunov function for complex-balanced
  mass-action systems}, arXiv preprint arXiv:1312.3043 (2013).

\bibitem{grasman1999asymptotic}
J.~Grasman and O.A. Herwaarden, \emph{Asymptotic methods for the fokker-planck
  equation and the exit problem in applications}, Springer Science \& Business
  Media, 1999.

\bibitem{herzog2014noise}
D.P. Herzog and J.C. Mattingly, \emph{Noise-induced stabilization of planar
  flows i}, arXiv preprint arXiv:1404.0957 (2014).

\bibitem{heymann2008geometric}
M.~Heymann and E.~Vanden-Eijnden, \emph{The geometric minimum action method: A
  least action principle on the space of curves}, Communications on pure and
  applied mathematics \textbf{61} (2008), no.~8, 1052--1117.

\bibitem{horn1972general}
F.~Horn and R.~Jackson, \emph{General mass action kinetics}, Archive for
  rational mechanics and analysis \textbf{47} (1972), no.~2, 81--116.

\bibitem{huang2013escape}
S.~Huang and S.~Kauffman, \emph{How to escape the cancer attractor: rationale
  and limitations of multi-target drugs}, Seminars in cancer biology, vol.~23,
  Elsevier, 2013, pp.~270--278.

\bibitem{huang5}
W.~Huang, M.~Ji, Z.~Liu, and Y.~Yi, \emph{Concentration and limit behaviors of
  stationary measures}, submitted (2015).

\bibitem{huang1}
\bysame, \emph{Integral identity and measure estimates for stationary
  fokker-planck equations}, Annals of Probability \textbf{43} (2015), no.~4,
  1712--1730.

\bibitem{huang2}
\bysame, \emph{Steady states of fokker-planck equations: I. existence}, Journal
  of Dynamics and Differential Equations, to appear (2015).

\bibitem{kifer1988random}
Y.~Kifer, \emph{Random perturbations of dynamical systems}, Birkh{\"a}user
  Boston, 1988.

\bibitem{kitano2004biological}
H.~Kitano, \emph{Biological robustness}, Nature Reviews Genetics \textbf{5}
  (2004), no.~11, 826--837.

\bibitem{kitano2007towards}
\bysame, \emph{Towards a theory of biological robustness}, Molecular systems
  biology \textbf{3} (2007), no.~1.

\bibitem{li2012quantification}
Y.~Li, G.~Dwivedi, W.~Huang, M.L. Kemp, and Y.~Yi, \emph{Quantification of
  degeneracy in biological systems for characterization of functional
  interactions between modules}, Journal of Theoretical Biology \textbf{302}
  (2012), 29--38.

\bibitem{lu2014construction}
M.~Lu, J.~Onuchic, and E.~Ben-Jacob, \emph{Construction of an effective
  landscape for multistate genetic switches}, Physical review letters
  \textbf{113} (2014), no.~7, 078102.

\bibitem{ludwig1975persistence}
D.~Ludwig, \emph{Persistence of dynamical systems under random perturbations},
  Siam Review (1975), 605--640.

\bibitem{lv2015energy}
C.~Lv, X.~Li, F.~Li, and T.~Li, \emph{Energy landscape reveals that the budding
  yeast cell cycle is a robust and adaptive multi-stage process}, PLoS
  computational biology \textbf{11} (2015), no.~3, e1004156--e1004156.

\bibitem{maier1993escape}
R.~Maier and D.~Stein, \emph{Escape problem for irreversible systems}, Physical
  Review E \textbf{48} (1993), no.~2, 931.

\bibitem{newby2011asymptotic}
J.~Newby and J.~Keener, \emph{An asymptotic analysis of the spatially
  inhomogeneous velocity-jump process}, Multiscale Modeling \& Simulation
  \textbf{9} (2011), no.~2, 735--765.

\bibitem{newby2014asymptotic}
Jay Newby, \emph{Asymptotic and numerical methods for metastable events in
  stochastic gene networks}, arXiv preprint arXiv:1412.8446 (2014).

\bibitem{newby2012isolating}
Jay~M Newby, \emph{Isolating intrinsic noise sources in a stochastic genetic
  switch}, Physical biology \textbf{9} (2012), no.~2, 026002.

\bibitem{pesin1997dimension}
Y.B. Pesin, \emph{Dimension theory in dynamical systems: contemporary views and
  applications}, University of Chicago Press, 1997.

\bibitem{ren2004minimum}
W.~Ren and E.~Vanden-Eijnden, \emph{Minimum action method for the study of rare
  events}, Communications on pure and applied mathematics \textbf{57} (2004),
  no.~5, 637--656.

\bibitem{risken1996fokker}
H.~Risken, \emph{The fokker-planck equation: Methods of solution and
  applications}, vol.~18, Springer Verlag, 1996.

\bibitem{rizk2009general}
A.~Rizk, G.~Batt, F.~Fages, and S.~Soliman, \emph{A general computational
  method for robustness analysis with applications to synthetic gene networks},
  Bioinformatics \textbf{25} (2009), no.~12, i169.

\bibitem{schuss2009theory}
Z.~Schuss, \emph{Theory and applications of stochastic processes: an analytical
  approach}, vol. 170, Springer Science \& Business Media, 2009.

\bibitem{talkner1987mean}
P~Talkner, \emph{Mean first passage time and the lifetime of a metastable
  state}, Zeitschrift f{\"u}r Physik B Condensed Matter \textbf{68} (1987),
  no.~2-3, 201--207.

\bibitem{tononi1994measure}
G.~Tononi, O.~Sporns, and G.M. Edelman, \emph{A measure for brain complexity:
  relating functional segregation and integration in the nervous system},
  Proceedings of the National Academy of Sciences \textbf{91} (1994), no.~11,
  5033.

\bibitem{tononi1999measures}
\bysame, \emph{Measures of degeneracy and redundancy in biological networks},
  Proceedings of the National Academy of Sciences of the United States of
  America \textbf{96} (1999), no.~6, 3257.

\bibitem{wang2008potential}
J.~Wang, L.~Xu, and E.~Wang, \emph{Potential landscape and flux framework of
  nonequilibrium networks: Robustness, dissipation, and coherence of
  biochemical oscillations}, Proceedings of the National Academy of Sciences
  \textbf{105} (2008), no.~34, 12271--12276.

\bibitem{wang2010potential}
J.~Wang, L.~Xu, E.~Wang, and S.~Huang, \emph{The potential landscape of genetic
  circuits imposes the arrow of time in stem cell differentiation}, Biophysical
  journal \textbf{99} (2010), no.~1, 29--39.

\bibitem{zhou2012quasi}
J.~Zhou, M.~Aliyu, E.~Aurell, and S.~Huang, \emph{Quasi-potential landscape in
  complex multi-stable systems}, Journal of The Royal Society Interface
  \textbf{9} (2012), no.~77, 3539--3553.

\bibitem{zhou2008adaptive}
X.~Zhou, W.~Ren, and E~Weinan, \emph{Adaptive minimum action method for the
  study of rare events}, The Journal of chemical physics \textbf{128} (2008),
  no.~10, 104111.

\end{thebibliography}

\end{document}